\documentclass[10pt,a4paper]{amsart}

\usepackage{amsmath}
\usepackage{amsthm}
\usepackage{enumerate} 
\usepackage{amssymb}
\usepackage{wasysym}
\usepackage{graphicx}
\usepackage[all]{xy}
\usepackage{hyperref}
\usepackage{aliascnt}
\usepackage{stmaryrd} 
\usepackage{mathtools}
\usepackage{verbatim} 
\usepackage{txfonts} 
\usepackage{lmodern}
\usepackage{setspace}   

\newaliascnt{thmCt}{lma}
\newtheorem{thm}[thmCt]{Theorem}
\aliascntresetthe{thmCt}

\newaliascnt{corCt}{lma}
\newtheorem{cor}[corCt]{Corollary}
\aliascntresetthe{corCt}

\newaliascnt{propCt}{lma}
\newtheorem{prop}[propCt]{Proposition}
\aliascntresetthe{propCt}

\newtheorem*{thm*}{Theorem}
\newtheorem*{dfn*}{Definition}
\newtheorem*{cor*}{Corollary}
\newtheorem*{prop*}{Proposition}
\newtheorem*{conj*}{Conjecture}

\theoremstyle{definition}

\newaliascnt{prgCt}{lma}
\newtheorem{prg}[prgCt]{}
\aliascntresetthe{prgCt}

\newaliascnt{dfnCt}{lma}
\newtheorem{dfn}[dfnCt]{Definition}
\aliascntresetthe{dfnCt}

\newaliascnt{rmkCt}{lma}
\newtheorem{rmk}[rmkCt]{Remark}
\aliascntresetthe{rmkCt}

\newaliascnt{rmksCt}{lma}

\aliascntresetthe{rmksCt}

\newaliascnt{ntnCt}{lma}

\aliascntresetthe{ntnCt}

\newaliascnt{ntnsCt}{lma}

\aliascntresetthe{ntnsCt}

\newaliascnt{qstCt}{lma}
\newtheorem{qst}[qstCt]{Question}
\aliascntresetthe{qstCt}

\newaliascnt{prblCt}{lma}

\aliascntresetthe{prblCt}

\newaliascnt{exaCt}{lma}

\aliascntresetthe{exaCt}

\newcommand{\CC}{\mathcal{C}}

\newcommand{\T}{\mathbb{T}}

\newcommand{\N}{\mathbb{N}}

\newcommand{\Z}{\mathbb{Z}}
\newcommand{\K}{\mathrm{K}}

\DeclareMathOperator{\im}{im}

\DeclareMathOperator{\Fun}{Fun}

\newcommand{\alg}{\mathrm{alg}}

\newcommand{\CatCa}{C^*}

\DeclareMathOperator{\Lat}{Lat}

\DeclareMathOperator{\AbGp}{AbGp}
\DeclareMathOperator{\Gp}{Gp}
\DeclareMathOperator{\PoS}{PoSet}
\DeclareMathOperator{\Cu}{Cu}

\DeclareMathOperator{\NCCW}{NCCW}
\DeclareMathOperator{\CW}{CW}
\DeclareMathOperator{\AI}{AI}
\DeclareMathOperator{\ASH}{ASH}
\DeclareMathOperator{\A}{A}

\DeclareMathOperator{\AH}{AH}
\DeclareMathOperator{\Ell}{Ell}
\DeclareMathOperator{\AF}{AF}
\DeclareMathOperator{\KT}{KT}
\DeclareMathOperator{\LF}{LF}
\DeclareMathOperator{\L(F}{L(F}
\DeclareMathOperator{\Lsc}{Lsc}

\DeclareMathOperator{\Hom}{Hom}
\DeclareMathOperator{\id}{id}

\DeclareMathOperator{\PoM}{PoM}
\DeclareMathOperator{\oM}{oM}
\newcommand{\PoSG}{\PoS_{*\!\Gp}}
\newcommand{\PoMG}{\PoM_{*{\Gp}}}

\newcommand{\CuG}{\Cu_{*\!\Gp}}
\newcommand{\CG}{\CC_{*{\Gp}}}
\newcommand{\Cus}{\Cu^*}
\newcommand{\oMs}{\oM^*}

\begin{document}
\onehalfspacing
\title{A systematic approach for invariants of $\CatCa$-algebras}
\author{Laurent Cantier}

\address{Laurent Cantier\newline
Departament de Matem\`{a}tiques \\
Universitat Aut\`{o}noma de Barcelona \\
08193 Bellaterra, Spain\newline
Institute of Mathematics\\ Czech Academy of Sciences\\ Zitna 25\\ 115 67 Praha 1\\ Czechia}
\email[]{laurent.cantier@uab.cat}

\thanks{The author was supported by the Spanish Ministry of Universities and the European Union-NextGenerationEU through a Margarita Salas grant and partially supported by the project PID2020-113047GB-I00.}

\begin{abstract} We define a categorical framework in which we build a systematic construction that provides generic invariants for $\CatCa$-algebras. The benefit is significant as we show that any invariant arising this way automatically enjoys nice properties such as continuity, metric on morphisms and a theory of ideals and quotients which naturally encapsulates compatibility diagrams. Consequently, any of these invariants appear as good candidates for the classification of non-simple $\CatCa$-algebras. Further, it is worth mentioning that most of the existing invariants could be rewritten via this method. As an application, we define an Hausdorffized version of the unitary Cuntz semigroup and explore its potential towards classification results. We pose several open lines of research.
\end{abstract}
\maketitle

\section{Introduction}
The classification of $\CatCa$-algebras is an area of research that has been around for decades and has been very prolific and productive. The field has been successful in many ways: it has built bridges between unrelated theories, obtained milestone results and continues to animate many researchers around the world. A fruitful modern approach consists in using functors as invariants, as they allow to classify not only the $\CatCa$-algebras themselves, but also morphisms between them. This approach gives concise, clear and powerful results. Among these, it is relevant to mention the classification of Kirchberg algebras via KK-Theory (see \cite{KP00},\cite{P07}) and the classification of simple nuclear $\CatCa$-algebras (see e.g. \cite{R02},\cite{W18} for a general overview). However, among all the invariants that have been employed, starting with the now-famous Elliott invariant up to its most complete form $\underline{\KT}_u$ without forgetting filtered $\K$-Theory for graph $\CatCa$-algebras, a lack of thoroughness of the categorical framework in which they operate can often be observed. The reasons for this fact are multiple, e.g. a clear categorical context is not necessary to obtain strong classification results. Also, most of sophisticated invariants are often the result of successive refinements over time done by patching existing ones together. Therefore the codomain category in which they dive in is often needed long lists of compatibility conditions between the objects and maps that come into play. This already happens in the simple setting and it gets even more pronounced in the non-simple setting since compatibility diagrams between ideals also need to appear. As a result, invariants for $\CatCa$-algebras enjoy a poor categorical context. Several consequences derive from this, e.g. the notion of abstract morphism in the codomain category is never obvious and the abstract study of the category itself becomes irrelevant. 
One of the only invariants to date which has a properly defined codomain category, for both simple and non-simple $\CatCa$-algebras, is the Cuntz semigroup. Let alone that it has provided tremendous classification results in the non-simple setting (see \cite{R12}), the study of abstract Cuntz semigroups has become an attractive area of work in its own right, independently of the classification of $\CatCa$-algebras, and has yielded a vast area of research that has many connections and applications with the rest of mathematics and still has a lot to unravel. See e.g. \cite{APRT22},\cite{APT18},\cite{APT20},\cite{APT20c}.

Driven by the goal of classifying non-simple $\CatCa$-algebras, the approach described above, (i.e. extending invariants for simple $\CatCa$-algebras to the non-simple setting by applying the invariant on all the ideals) is not satisfactory from a categorical viewpoint since it imposes to \textquoteleft manually\textquoteright\ add compatibility conditions via complex commutative diagrams. Another approach consists in defining a variation of the Cuntz semigroup (which has been proven a solid invariant when it comes to unravel the ideal structure of the $\CatCa$-algebra) that incorporates  information not carried by the Cuntz semigroup itself. The present author has been constructing a unitary Cuntz semigroup which incorporates $\K_1$-information into the Cuntz semigroup in a satisfactory manner as far as the categorical framework is concerned. (The codomain category $\Cu^\sim$ is slightly larger but roughly similar to the category $\Cu$. See \cite{C21a}, \cite{C21b}) Subsequently, another study has sought to mimic this approach in order to define a total Cuntz semigroup which incorporates the total $\K$-Theory into the Cuntz semigroup. See \cite{AL23}. Both of the constructions are very alike, share many categorical properties and even though many proofs were based on similar arguments, they had to be done twice with equal care. 

The main motivation of this paper lies in using Category Theory to abstractly (and hence systematically) generalize this approach. In particular, we expose a generic construction which produces variations of the Cuntz semigroup which incorporate (almost) any additional information one desires together with a unified categorical framework, modeled after the category of abstract Cuntz semigroups. Therefore many of the existing (and hopefully upcoming) invariants could properly fit into this category when redefined via the generic method we provide. In a first part, we set up a categorical context by defining \emph{Cuntz group systems} which yields the definition of the category $\Cus$. 
\begin{dfn*}
The category $\Cus$ is the category whose objects are ordered monoids with a compact neutral element, satisfying order-theoretic axioms (O1)-...-(04) together with axioms (PC), (PD), (S0) and whose morphisms are ordered monoid morphisms preserving the compact-containment relation and suprema of $\leq$-increasing sequences. 
\end{dfn*}
Subsequently, we investigate the category $\Cus$. We show that it has inductive limits and we develop a theory of ideals and quotients for abstract $\Cus$-semigroups. We also expose a generic method, that we call the \emph{$\Cu_\K$-type construction}, to build invariants for $\CatCa$-algebras suitable for the non-simple setting. We show that these variations of the Cuntz semigroup (including the unitary and the total Cuntz semigroup) enjoy a clear categorical framework together many categorical properties such as continuity, exactness, an ideal-quotient theory and distances on morphism. We are particularly interested in a variation of the Cuntz semigroup incorporating the Hausdorffized algebraic $\K_1$-group termed $\overline{\K}^{\alg}_1$. (See \cite{T95} and \cite{NT96}.) This ingredient has played a key role in the classification of simple $\A\!\T$-algebras done by Nielsen and Thomsen in \cite{NT96}. In particular, $\overline{\K}^{\alg}_1$ is intimately related to the de la Harpe-Skandalis determinant, which appears to be necessary when it comes to classifying unitary elements, as pointed out in \cite{C23b}. Consequently, we define a Hausdorffized unitary Cuntz semigroup via our method.
\begin{thm*}
The Hausdorffized unitary Cuntz semigroup $\Cu_{\overline{\K}^{\alg}_1}: \CatCa\longrightarrow \Cus$
is a well-defined continuous functor merging the Hausdorffized algebraic $\K_1$-group into the Cuntz semigroup.
\end{thm*}

Lastly, we pose several open lines of researches including proposals for other generic methods that could produce other types of invariants for $\CatCa$-algebras, a metric on the set of concrete $\Cus$-morphisms and hints towards classification results in the non-simple setting. In particular, we state a conjecture outlined in \cite[Section 6]{C23b}. By $\AH_1$-algebras, we mean inductive limits of (direct sums of) homogeneous $\CatCa$-algebras with one-dimensional spectrum and by $\ASH_1$-algebras, we mean inductive limits of (direct sums of) one-dimensional $\NCCW$-complexes. Note that any $\AH_1$-algebra is an $\ASH_1$-algebra with torsion-free $\K_1$-group. E.g, $\AF,\AI,\A\!\T$-algebras are $\AH_1$-algebras while the Jiang-Su algebra $\mathcal{Z}$ is an $\ASH_1$-algebra.

\begin{conj*} The Hausdorffized unitary Cuntz semigroup classifies unitary elements of any unital $\ASH_1$-algebra with torsion-free $\K_1$-group. 
\end{conj*}

\textbf{Acknowledgments.} 
The author would like to thank the people of the Czech Academy of Sciences for their warm hospitality, flexibility and mathematical versatility, providing a great working environment. He is also grateful to the referee for her/his comments.

\section{Group Systems}
We first define a categorical context which allows to merge a partially ordered set $S$ with a functor $G:S\longrightarrow \AbGp$, where $\AbGp$ denotes the category of abelian groups. Roughly speaking, $S$ that will act as a \textquoteleft base space\textquoteright\ and the functor $G:S\longrightarrow \AbGp$ will assign a \textquoteleft fiber\textquoteright\ $G(s)\in \AbGp$ to each point $s\in S$. The partial order on $S$ gives a relation between the fibers via group morphisms encapsulated within the functoriality of $G$. This \textquoteleft web\textquoteright\ of abelian groups and morphisms is referred to as a \emph{system of group over $S$}. Subsequently, we are able to construct a partially ordered set $S_G$ whose elements are pairs $(s,g)$ where $s\in S$ and $g\in G(s)$, via the \emph{webbing transformation}. In addition, most of the extra-structures and extra-properties, such as a sum or order-theoretic axioms, pass from $S$ to $S_G$. This allows us to define 3 categories of group systems: the category of partially ordered group systems, the category of pomonoid group systems and finally the category of Cuntz group systems.

\subsection{Categories of group systems and webbing functors}

We first introduce the general notion of \emph{systems of groups}. We recall that a partially ordered set $S$ can be viewed as a small category, also written $S$, whose objects are the elements $s$ of $S$ and the arrows are induced by the order. In other words, there is an arrow from $s$ to $t$ that we write $s\overset{\leq}\longrightarrow t$ if (and only if) $s\leq t$ in $S$.

\begin{dfn}
Let $(S,\leq)$ be a partially ordered set. 

(i) For each $s$, we consider an (abelian) group $G_s$ associated to $s$.

(ii) For each $s,t\in S$ with $s\leq t$, we consider a group morphism $\phi_{st}:G_s\longrightarrow G_t$ with the convention that $\phi_{ss}=id_{G_s}$.

The pair $(\{G_s\},\{\phi_{st}\})_S$ is called a \emph{system of groups over $S$}. \\
Note that a system of groups over $S$ can be seen as the image $G(S)$ of a functor $G:S\longrightarrow \AbGp$ that sends $s\longmapsto G_s$ and $(s\overset{\leq}\longrightarrow t)\longmapsto \phi_{st}$.
\end{dfn}

$\hspace{-0,34cm}\bullet\,\,\textbf{The category}\PoSG \textbf{of partially ordered group systems}$. We denote the category of partially ordered sets by $\PoS$. We focus on the case where the base space $S\in\PoS$.

A \emph{partially ordered group system} is a pair $(S,G)$ where $S\in \PoS$ and $G\in\Fun(S,\AbGp)$.
 
An \emph{order-preserving map} between partially ordered group systems is a pair $(\alpha,\eta):(S,G)\longrightarrow (T,H)$ where $\alpha:S\longrightarrow T\in\Hom_{\PoS}(S,T)$ and $\eta:G\Rightarrow H\circ\alpha$ is a natural transformation. (Remark that $\alpha$ is viewed as a functor between the small categories $S$ and $T$.)
We now define the category of partially ordered group systems as follows.

\begin{dfn}
The category $\PoSG$ of \emph{ordered group systems} consists of partially ordered group systems together with order-preserving maps, where the composition of morphism is given as follows.

Let $(\alpha,\eta):(S,G)\longrightarrow (T,H)$ and $(\beta,\nu):(T,H)\longrightarrow (V,J)$ be $\PoSG$-morphisms. Observe that $\nu$ naturally induces a natural transformation $\nu_\alpha:H\circ\alpha\longrightarrow J\circ \beta\circ\alpha$ given by $(\nu_\alpha)_s:=\nu_{\alpha(s)}$. We now define $(\beta,\nu)\circ(\alpha,\eta):=(\beta\circ\alpha, \nu_\alpha\circ \eta)$, where $\nu_\alpha\circ \eta$ is the usual vertical composition of transformations. We sum up this composition of transformation in the following commutative diagram:
\[
\xymatrix@!R@!C@R=40pt@C=0pt{
G(s)\ar@{-->}@/^{-3pc}/[dd]_{(\nu_\alpha\circ \eta)_s}\ar[rr]_{G(\leq)}\ar[d]^{\eta_s} && G(t)\ar@{-->}@/_{-3pc}/[dd]^{(\nu_\alpha\circ \eta)_t}\ar[d]_{\eta_t} \\
H\circ\alpha(s)\ar[rr]_{H\circ\alpha(\leq)}\ar[d]^{\nu_{\alpha(s)}} && H\circ\alpha(t)\ar[d]_{\nu_{\alpha(t)}}\\
J\circ\beta\circ\alpha(s)\ar[rr]_{J\circ\beta\circ\alpha(\leq)}&& J\circ\beta\circ\alpha(t)
}
\]
where $\leq$ stands for the (unique) $S$-morphism $s\overset{\leq}\longrightarrow t$. 
\end{dfn}

We mention that the above terminology is somewhat inspired by the notion of slice categories.
Now that we have a proper categorical context, from an ordered system of group $(S,G)$ we wish to build a poset $S_G$ which is based on $S$ but also enriched with the extra-data contained in $G$. Namely, let $(S,G)\in\PoSG$. We define 
\[
S_G:=\{(s,g)\mid s\in S, g\in G(s)\}
\]
We equip $S_G$ with the following partial order:  
\[(s,g)\leq (t,h) \text{ if, } s\leq t \text{ and } G(s\leq t)(g)=h.
\]
Let $(S,G)\overset{(\alpha,\eta)}\longrightarrow (T,H)$ be a $\PoSG$-morphism. We define
 \[
\begin{array}{ll}
\alpha_\eta:S_G\longrightarrow T_H\\
\hspace{0,35cm} (s,g)\longmapsto (\alpha(s),\eta_s(g))
\end{array}
\]
The above constructions allow us to define a \emph{webbing functor}

\begin{prop}
The assignment \[
\begin{array}{ll}
 \mathfrak{w}: \PoSG\longrightarrow \PoS \\
\hspace{1,05cm} (S,G)\longmapsto S_G\\
		\hspace{1,15cm}(\alpha,\eta) \longmapsto \alpha_{\eta}
\end{array}
\] 
is a well-defined functor called the webbing functor. 
\end{prop}

\begin{proof}
The fact that $S_G$ is a partially ordered set and $\alpha_\eta$ is an order-preserving map is straightforward and left to the reader. We have to prove that $\mathfrak{w}$ preserves the composition of morphisms. 
Let $(S,G)\overset{(\alpha,\eta)}\longrightarrow (T,H)\overset{(\beta,\nu)}\longrightarrow (V,J)$ be $\PoSG$-morphisms. It is enough to show that $\beta_\nu\circ\alpha_\eta=(\beta\circ\alpha)_{(\nu_\alpha\circ\eta)}$. 
Let $(s,g)\in S_G$. We compute that 
\begin{align*}
\beta_\nu\circ\alpha_\eta(s,g)&=\beta_\nu(\alpha(s),\eta_s(g))\\
&=(\beta\circ\alpha(s),\nu_{\alpha(s)}(\eta_s(g))\\
&=(\beta\circ\alpha(s),(\nu_\alpha\circ\eta)_s(g))\\
&=(\beta\circ\alpha)_{(\nu_\alpha\circ\eta)}(s,g).
\end{align*}
\end{proof}

$\hspace{-0,34cm}\bullet\,\,\textbf{The category}\PoMG \textbf{of pomonoid group systems and the category} \oMs$. 
We denote the category of positively ordered monoid (pomonoids in short) by $\PoM$.
We focus on the case where the base space $S\in\PoM$ and investigate to what extent the monoid structure passes onto $S_G$. Mimicking the above, we define the following.

A \emph{pomonoid group system} is a pair $(S,G)$ where $S\in\PoM$ and $G\in\Fun(S,\AbGp)$ such that $G(0_S)\simeq\{e\}$.

A \emph{pomonoid group system map} is a pair $(\alpha,\eta):(S,G)\longrightarrow (T,H)$ where $\alpha:S\longrightarrow T\in\Hom_{\PoM}(S,T)$ and $\eta:G\Rightarrow H\circ\alpha$ is a natural transformation. 

Here again, $S$ and $T$ are viewed as small categories and $\alpha$ as a functor. Also, observe that that we have required the extra-condition $G(0_S)\simeq\{e\}$. (For technical reasons.)

\begin{dfn}
The category $\PoMG$ of \emph{pomonoid group systems} consists of pomonoid group systems and pomonoid group system maps, where the composition of morphisms is given as before.
\end{dfn}

Let $(S,G)\in \PoMG$. We equip the partially ordered set $S_G$ with the following sum 
\[(s,g)+(t,h):= (s+t, G(s\leq s+t)(g)+G(t\leq s+t)(h)).\]
\begin{rmk}
Positiveness  of the partial order in $S$ is crucial. More specifically, without positiveness one cannot ensure that $s\leq s+t$ (respectively $t\leq s+t$) and the above sum may not be well-defined on $S_G$. 
\end{rmk}
The next proposition shows that $S_G$ is an ordered monoid whenever $S$ is a pomonoid. It is worth mentioning that $S_G$ is not a pomonoid in general, but the partial order on $S_G$ still satisfies a weak form of positiveness. More specifically, $S_G$ satisfies axioms (PC) and (PD) introduced in \cite{C21b} (which roughly state that there are no negative elements and any non-positive element has a \textquoteleft symmetric\textquoteright\ with respect to the positive cone) together with an axiom (S0).

\begin{dfn}\cite[Definition 3.3]{C21b}
Let $S$ be an ordered monoid. We define the following axioms.

(PD) We say that $S$ is \emph{positively directed} if, for any $s\in S$, there exists $p_s\in S$ such that $s+p_s\geq 0$.

(PC) We say that $S$ is \emph{positively convex} if, for any $s,t\in S$ such that $t\geq 0$ and $s\leq t$, we have $s+t\geq 0$.

(S0) We say that $S$ is \emph{singular at the origin} if, $s+t= 0$ implies $s=t=0$.
\end{dfn}

\begin{prop}
Let $(S,G)\in \PoMG$. Then 

(i) $S_G$ is an ordered monoid.

(ii) $S_G$ satisfies axioms (PC), (PD) and (S0).
\end{prop}

\begin{proof}
(i) First, it can be observed that $(0,e)$ is the neutral element of $S_G$. Next, we have to check that the order and the sum are compatible. Let $(s,g),(t,h)$ and $(s',g'),(t',h')$ be such that $(s,g)\leq(t,h)$ and $(s',g')\leq(t',h')$. Then $s+s'\leq t+t'$ in $S$. Moreover, by the functoriality of $G$, we know that $G(s+s'\leq t+t')\circ G(s\leq s+s')=G(t\leq t+t')\circ G(s\leq t)$, respectively for $s'$ and $t'$. This allows us to compute that $G(s\leq t+t')(g)+G(s'\leq t+t')(g')=G(t\leq t+t')(h)+G(t'\leq t+t')(h')$, from which the result follows.

(ii) (PD) Let $(s,g)\in S$. Then $(s,g)+(s,-g)=(2s,e_{G(2s)})\geq (0,e)$.

(PC) Let $(s,g),(t,h)$ be such that $(s,g)\leq(t,h)$ and $(0,e)\leq (t,h)$. The second inequality implies that $h=e_{G(t)}$. Now, observing that $G(t\leq s+t)\circ G(s\leq t)=G(s\leq s+t)$, we deduce that $G(s\leq s+t)(g)=G(t\leq s+t)(e_{G(t)})$ and hence, that $(s,g)+(t,h)=(s+t,e_{G(s+t)})$.

(S0) Let $(s,g),(t,h)$ be such that $(s,g)+(t,h)=(0,e)$. Then $s+t=0$ in $S$ which implies that $s=t=0$. Now (S0) follows from the fact that $G(0_S)$ is the trivial group. 
\end{proof}

As for partially ordered groups, we wish to build a suitable webbing functor for pomonoid group systems. Nevertheless, we have just pointed out that the webbing transformation turns a pomonoid group system into an ordered monoid which only have a weak form of positiveness. Thus, we first have to introduce the following category located between $\PoM$ and $\oM$, where $\oM$ denotes the category of ordered monoids.

\begin{dfn}
The category $\oM^*$ is the full subcategory of $\oM$ whose objects are ordered monoids satisfying axioms (PC), (PD) and (S0).
\end{dfn}

We now obtain a webbing functor for the category $\PoMG$ by restriction as follows.
\begin{prop}
The webbing functor $\mathfrak{w}: \PoMG\longrightarrow \oMs$ is well-defined.
\end{prop}

\begin{proof}
We leave the reader to check that $\alpha_\eta$ is a monoid morphism.
\end{proof}

$\hspace{-0,34cm}\bullet\,\,\textbf{The category}\CuG \textbf{of Cuntz group systems and the category}\Cus$. 
We denote the category of abstract Cuntz semigroups by $\Cu$.
Recall that the \emph{Cuntz semigroup} was first introduced in \cite{C78}. We refer the reader to \cite{GP23} for a complete survey about Cuntz semigroups where all the basics can be found. Another fundamental and complete paper can be found in \cite{APT18}. Finally, categorical aspects of the category $\Cu$ have been developed in \cite{APT20}, \cite{APT20b} and state-of-the-art results in \cite{APRT22}, \cite{T20}. Let us first briefly recall order-theoretic axioms satisfied by a $\Cu$-semigroup. 
\begin{prg}
Let $(S,\leq)$ be an ordered monoid and let $x,y$ in $S$. We say that $x$ is \emph{way-below} $y$ and we write $x\ll y$ if, for all increasing sequences $(z_n)_{n\in\N}$ in $S$ that have a supremum, if $\sup\limits_{n\in\N} z_n\geq y$, then there exists $k$ such that $z_k\geq x$. This is an auxiliary relation on $S$ called the \emph{way-below relation} or the \emph{compact-containment relation}. In particular $x\ll y$ implies $x\leq y$ and we say that $x$ is a \emph{compact element} whenever $x\ll x$. 

We say that $S$ is an \emph{abstract Cuntz semigroup}, or a $\Cu$-semigroup, if $S$ satisfies the following order-theoretic axioms: 

(O1) Every increasing sequence of elements in $S$ has a supremum. 

(O2) For any $x\in S$, there exists a $\ll$-increasing sequence $(x_n)_{n\in\N}$ in $S$ such that $\sup\limits_{n\in\N} x_n= x$.

(O3) Addition and the compact containment relation are compatible.

(O4) Addition and suprema of increasing sequences are compatible.

We say that a map $\alpha:S\longrightarrow T$ is a $\Cu$-morphism if $\alpha$ is an ordered monoid morphism preserving the compact-containment relation and suprema of increasing sequences.

\end{prg}

 We focus on the case where the base space $S\in \Cu$ and investigate to what extent the order-theoretic axioms pass onto $S_G$. Mimicking the above, we define the following.

A \emph{Cuntz group system} is a pair $(S,G)$ where $S\in\Cu$ and $G\in\Fun(S,\AbGp)$ such that $G$ preserves inductive limits and $G(0_s)\simeq\{e\}$.  

A \emph{Cuntz group system map} is a pair $(\alpha,\eta):(S,G)\longrightarrow (T,H)$ where $\alpha:S\longrightarrow T\in\Hom_{\Cu}(S,T)$ and $\eta:G\Rightarrow H\circ\alpha$ is a natural transformation. 

Here again, $S$ and $T$ are viewed as small categories and $\alpha$ as a functor. Also, observe that we have required the extra-condition that $G$ preserves inductive limits. (For technical reasons.)
 
 \begin{dfn}
The category $\CuG$ of \emph{Cuntz group systems} consists of Cuntz group systems and Cuntz group system maps, where the composition of morphisms is given as before.
\end{dfn}

We aim to prove that the webbing transformation gives us an ordered monoid $S_G$ satisfying the order-theoretic axioms (O1)-...-(O4) whenever $S\in \Cu$. First, we observe in the following lemma that the compact-containment relation in $S_G$ is entirely determined by the one in $S$.

\begin{prop}
\label{prop:cc}
Let $(s,g),(t,h)\in S_G$. Then the following are equivalent:

(i) $(s,g)\ll(t,h)$.

(ii)  $s\ll t$ and $(s,g)\leq(t,h)$.

(iii) $s\ll t$ and  $G(s\leq t)(g)=h$.
\end{prop}

\begin{proof}
(ii) implies (iii) is immediate. 

(i) implies (ii): Assume that $(s,g)\ll(t,h)$. Hence, for any increasing sequence $(z_n,l_n)_n$ in $S_G$ that has a supremum $(z,l)$ such that $(z,l)\geq (t,h)$, then $(z_n,l_n)\geq (s,g)$ for some $n\in\N$. In particular, $(s,g)\leq(t,h)$. To show that $s\ll t$, consider an increasing sequence $x_n$ in $S$ whose supremum $x\geq t$. The functor $G$ preserves inductive limits and hence induces the following inductive system in $\AbGp$ with limit $G(x)$.
\[
\xymatrix{
G(x_n)\ar[r]^{} &\dots\ar[r]^{} & G(x_m)\ar[r]^{}&\dots\ar[r]^{}  & G(x)
}
\]
whose limit is $(G(x),G(x_n\leq x))$. Further, limits in $\AbGp$ are algebraic and hence, for any element $k\in G(x)$, there exists some $k_n\in G(x_n)$ such that $k=G(x_n\leq x)(k_n)$. In particular, there exists some $k_m\in G(x_m)$ such that $k:=G(t\leq x)(h)=G(x_m\leq x)(k_m)$. Now define $(z_n,l_n):=(x_{m+n},G(x_m,x_{m+n})(k_m)$). We have constructed an increasing sequence $(z_n,l_n)_n$ in $S_G$ with supremum $(x,k)\geq (t,h)$. Applying the hypothesis, we deduce that $(z_n,l_n)\geq (s,x)$ for some $n\in\N$. A fortiori, $x_{m+n}=z_n\geq s$ in $S$ for some $n\in\N$.

(iii) implies (i): Let $(z_n,l_n)_n$ be an increasing sequence in $S_G$ that has a supremum $(z,l)$ such that $(z,l)\geq (t,h)$. A fortiori, $(z_n)_n$ is an increasing sequence in $S$ and $z\geq t$. In a first step, let us show that $z$ is the supremum of $(z_n)_n$ in $S$. It is clear that $z_n\leq z$ for all $n\in\N$. Let $y\in S$ be such that $z_n\leq y$ for all $n\in\N$. Consider $m:=G(z_n\leq y)(l_n)=G(z_m\leq y)(l_m)$. By construction, we have $(z_n,l_n)\leq (y,m)$ for all $n\in\N$ and hence, $(z,l)\leq (y,m)$. This implies that $z\leq y$ and we conclude that $z$ is the supremum of $(z_n)_n$.
Next, since $s\ll t\leq z$, we can find some $n\in\N$ such that $z_n\geq s$. Therefore, using again that $G$ preserves inductive limits, we obtain the following commutative diagram
\[
\xymatrix{
G(s)\ar[rrrr]^{}\ar[d]_{} & & & &G(t)\ar[d]^{}\\
G(z_n)\ar[r]^{} &\dots\ar[r]^{} & G(z_m)\ar[r]^{}&\dots\ar[r]^{}  & G(z)\\
}
\]
Thus $G(s\leq z)(g)=G(z_n\leq z)(l_n)$. Again, inductive limits in $\AbGp$ are algebraic, from which we deduce that there exists some $m\in\N$ such that $G(s\leq z_m)(g)=G(z_n\leq z_m)(l_n)=l_m$. This terminates the proof.
\end{proof}

We note that the construction done in the argument proving (i) implies (ii) a fortiori shows that $S_G$ satisfies axiom (O2). We give some more details in the next proof. 

\begin{prop}
Let $(S,G)\in \CuG$. Then $S_G$ satisfies order-theoretic axioms (O1)-...-(04) and its neutral element $(0,e)$ is compact. 
\end{prop}

\begin{proof}
The fact that $(0,e)$ is compact falls immediately from the previous proposition.

(O1) Let $(s_n,g_n)_n$ be an increasing sequence in $S_G$. Then $(s_n)_n$ is an increasing sequence and its supremum, that we write $s$, exists in $S$.  Also, observe that for any $m\in\N$, we have $G(s_m\leq s_{m+1})(g_m)=g_{m+1}$, from which we deduce that $g:=G(s_m\leq s)(g_m)=G(s_{m+1}\leq s)(g_{m+1})$. By construction, $(s,g)\geq (s_m,g_m)$ for any $m$. Now let $(t,h)\in S_G$ be such that $(t,h)\geq (s_m,g_m)$ for any $m$. By the definition of the supremum, we have that $t\geq s$ in $S$. Further, $G(s_m\leq t)(g_m)=G(s\leq t)\circ G(s_m\leq s)(g_m)=G(s\leq t)(g)$. Therefore $(s,g)\leq (t,h)$.

(O2) Let $(s,g)\in S_G$. Consider a $\ll$-increasing sequence $(s_n)_n$ in $S$ whose supremum is $s$, given by (O2). Now arguing similarly as in the above proof, we can construct an increasing sequence $(s_{n+m},g_{n+m})_n$ whose supremum is $(s,g)$. Finally, using the equivalences of the above proposition, we conclude that the latter sequence is in fact $\ll$-increasing. 

(O3) Immediate from the previous proposition.

(O4) Let $(s_n,g_n)_n, (t_n,h_n)_n$ be increasing sequences whose suprema are $(s,g),(t,h)$ respectively. Then $(s_n+t_n)_n$ is an increasing sequence in $S$ whose supremum is $s+t$ and the following diagram is commutative in $\AbGp$
\[
\xymatrix{
G(s_n)\ar[r]^{}\ar[d]_{} & G(s_n+t_n)\ar[d]_{} & G(t_n)\ar[l]^{}\ar[d]^{}\\
G(s)\ar[r]^{} & G(s+t) & G(t)\ar[l]^{}\\
}
\]
Thus $G(s_n+t_n\leq s+t)(l_n)=g+h$, where $l_n:=G(s_n\leq s_n+t_n)(g_n)+G(t_n\leq s_n+t_n)(h_n)$. We conclude that $\sup (s_n,g_n)+ \sup (t_n,h_n)=\sup(s_n+t_n,l_n)=\sup((s_n,g_n)+(t_n,h_n))$.
\end{proof}

As for pomonoid group systems, we have to introduce the following category located between $\Cu$ and $\Cu^\sim$ introduced in \cite{C21a}, in order to build a webbing functor for Cuntz group systems.

\begin{dfn}
The category $\Cus$ is the category whose objects are ordered monoids with a compact neutral element, satisfying order-theoretic axioms (O1)-...-(04) together with axioms (PC), (PD), (S0) and whose morphisms are ordered monoid morphisms preserving the compact-containment relation and suprema of $\leq$-increasing sequences. 
\end{dfn}

We highlight that $\Cus$ can equivalently be seen as the full subcategory of $\Cu^\sim$ introduced in \cite{C21a}, whose objects are $\Cu^\sim$-semigroups satisfying axioms (PC), (PD) and (S0). We now obtain a webbing functor for the category $\CuG$ by restriction as follows.

\begin{prop}
The webbing functor $\mathfrak{w}: \Cu_{*\!\Gp}\longrightarrow \Cu^*$ is well-defined.
\end{prop}

\begin{proof}
We leave the reader to check that $\alpha_\eta$ preserves $\ll$ and suprema of increasing sequences.
\end{proof}

It is worth mentioning that the proofs of (O1) and (O2) could be applied to look at other categories in domain theory. For instance, it would be easily checked that the webbing transformation of a partially ordered group system $(S,G)$ where $S$ is in fact a sequentially directed complete poset (i.e. satisfying (O1)) gives a sequentially directed complete poset. If moreover $S$ is a sequentially continuous (i.e. satisfying (O2)), then so is $S_G$.

We now have a look at additional axioms that have been considered for $\Cu$-semigroups and we study whether they pass through the webbing functor (even in a weaker form). Let us recall some of these axioms and we refer the reader to \cite{R13}, \cite{RW10} and \cite{C23} for more details. (Also, \cite[Section 10]{GP23}.) Let $S$ be a $\Cu$-semigroup. We say that $S$ satisfies

(Weak Cancellation) if, whenever we have $s+t\ll v+ t$ then $s\ll v $.

(O5) if, whenever we have $s\leq t$ and $s'\ll s$ then there exists $w\in S$ such that $s'+w\leq t\leq s+w$.

(O6) if, whenever we have $x'\ll x\leq y+z$, then there exist $s\leq x,y$ and $t\leq x,z$ such that $x'\leq s+t$.\\
Let $S$ be a $\Cus$-semigroup. We say that $S$ satisfies

(Positive Weak Cancellation) if, whenever we have  $s+t\ll t$ then $s\ll 0$.

\begin{prop}
\label{prop:webprop}
Let $(S,G)\in\CuG$.

(i) If $S$ satisfies (WC) then $S_G$ satisfies (PWC). 

(ii) If $S$ satisfies (O5) then does $S_G$.

(iii) There exists a $\CatCa$-algebra $A$ and a functor $G\in \Fun(\Cu(A),\AbGp)$ such that $\mathfrak{w}(\Cu(A),G)$ does not satisfy (O6). 

(iv) If $S$ is almost divisible and $G$ is stable, in the sense that $G(s\leq k.s)$ is a group isomorphism for all $k\in\overline{\N}$, then $S_G$ is almost divisible.

(v) There exist a $\CatCa$-algebra $A$ and a functor $G\in \Fun(\Cu(A),\AbGp)$ such that $\Cu(A)$ is almost unperforated and yet $\mathfrak{w}(\Cu(A),G)$ is not almost unperforated. 
\end{prop}

\begin{proof} All the following proofs are based on the characterization of the compact-containment relation in $S_G$ given in \autoref{prop:cc}.

(i) Assume that $S$ satisfies (WC) and let $(s,g),(t,h)$ in $S_G$ be such that $(s,g)+(t,h)\ll (t,h)$. This implies that $s+t\ll t$ in $S$. Thus $s\ll 0$. It follows from $G(0)=\{e\}$ that $(s,g)= 0_{S_G}\ll 0_{S_G}$. 

(ii) Assume that $S$ satisfies (O5) and let $(s',g'),(s,g),(t,h)$ in $S_G$ be such that $(s',g')\ll (s,g)\leq (t,h)$. This implies that $s'\ll s\leq t$ in $S$. Thus we can find some $w\in S$ such that $s'+w\leq t \leq s+w$. Moreover, the assumption also implies that $G(s'\leq t)(g')=h=G(s\leq t)(g)$. Adding $G(w\leq t)(0)$ on both sides, we obtain that $G(s'\leq t)(g')+G(w\leq t)(0)=h=G(s\leq t)(g)+G(w\leq t)(0)$. We deduce that $G(s'+w\leq t)[G(s'\leq s'+w)(g')+G(w\leq s'+w)(0)]=h$ and that $G(t\leq s+w)(h)=G(s\leq s+w)(g)+G(w\leq s+w)(0)$. We conclude that $(s',g')+(w,0)\leq (t,h)\leq (s,g)+(w,0)$ and hence, $S_G$ satisfies (O5). 

(iii) Let $A:=\CC(\T^2)$ and consider the functor $G:\Cu(A)\longrightarrow \AbGp$ sending $x\longmapsto\K_1(I_x)$, where $I_x$ is the ideal of $A$ generated by $x$. (We are using the lattice isomorphism $\Lat(A)\simeq \Lat(\Cu(A))$.) Then $\mathfrak{w}(\Cu(A),G)$ is the unitary Cuntz semigroup of $A$ that we write $\Cu_1(A)$. Consider $x',x,y,z$ be in $\Cu(A)$ such that they are all equal to $[1_A]$. Then $x'\ll x\leq y+z$ and these are full elements in $A$. Also, recall that $\K_1(A)\simeq \Z^4$ and that for any connected open set $U\subsetneq \T^2$ and positive element $f_U\in A$ with support $U$, we have $\K_1(I_{[f_U]})\simeq \K_1(C_0(R^2))\simeq \{e\}$. Now, using (O6) in $\Cu(A)$, we can find positive elements $a,b$ in $A\otimes\mathcal{K}$ such that $[a]\leq x,y$ and $[b]\leq x,z$ and $x'\leq [a]+[b]$. As a result, we deduce that $a,b$ belongs to $A$ and that they have open supports $U,V$ satisfying $U,V\subseteq \T^2$ and $U\cup V\supseteq \T^2$. Thus, the groups $\K_1(I_{[a]}),\K_1(I_{[b]})$ are either trivial or isomorphic to $\Z^4$. Now consider the following elements in $\Cu_1(A)$
\[
(x',(1,1,0,0))\ll (x,(1,1,0,0)) \leq (y,(1,0,0,0))+ (z,(0,1,0,0))
\]
Observe that $([a],e)$ does not compare neither with $(x,(1,1,0,0))$ nor $(y,(1,0,0,0))$. Thus if $\K_1(I_{[a]})$ is trivial, we cannot find the required elements. On the other hand, if $\K_1(I_{[a]})\simeq \Z^4$ (or equivalently, if $[a]=[1_\T]$) we have to find $([a],k)$ such that $([a],k)\leq (x,(1,1,0,0)), (y,(1,0,0,0))$ and hence, $k$ has to be equal to $(1,1,0,0)$ and $(1,0,0,0)$ simultaneously which is a contradiction. 

(iv) Assume that $S$ is almost divisible. Let $(s',g'),(s,g)$ in $S_G$ be such that $(s',g')\ll (s,g)$ and let $n\in\N$. This implies that $s'\ll s$ and that $G(s'\leq s)(g')=g$. Thus, there exists $w\in S$ such that $nw\leq s$ and $s'\leq (n+1)w$. Now consider $h:=G(s'\leq (n+1)w)(g')$. We know from the stability of $G$ that $G(nw\leq (n+1)w)$ is an isomorphism, thus we can find a unique $\hat{h}$ in $G(nw)$ satisfying $G(nw\leq s)(\hat{h})=g$, from which the almost divisibility of $S_G$ follows.

(v) Let $A:=\CC(\T)$ and consider the functor $G:\Cu(A)\longrightarrow \AbGp$ sending $x\longmapsto\K_1(I_x)$. Recall that $\K_1(A)\simeq \Z$. Let $n\in\N$. Observe that for any $k\in \Z\setminus\{0\}$ we have that $(n+1)([1_A],nk)\leq n(2[1_A],(n+1)k)$ and obviously, $([1_A],nk)$ does not compare with $([1_A],(n+1)k)$.
\end{proof}

\begin{qst} Are weaker notions of almost unperforation and (O6) suitable for $\Cus$-semigroups, in the sense that if $S$ satisfies the latter properties then $S_G$ satisfies their weaker version?
\end{qst}

\subsection{Continuity of the webbing functors}
We aim to show that all the webbing functors defined above are continuous. To that end, we first show the existence of inductive limits in the involved domain and codomain categories. In the sequel, $\CC$ stands for either $\PoS, \PoM$ or $\Cu$.

\vspace{0,2cm}$\hspace{-0,34cm}\bullet\,\,\textbf{Inductive limits in the categories }\CC_{*{\Gp}}$. The main idea is to exhibit a natural candidate for the $\CC_{*{\Gp}}$-limit based on the $\CC$-limit and show that it satisfies the universal property. We highlight that an extra-effort is needed for $\CuG$, as one needs to take care of a topological aspect. 
Let us first recall details on the composition of $\CC_{*{\Gp}}$-morphisms as a clarifying reminder.

Let $( (S_i,G_i), (\alpha_{ij}, \eta_{ij}) )_{i\in I}$ be inductive system in $\CC_{*{\Gp}}$. For any $i\leq j\leq k$, we know that $(\alpha_{jk}, \eta_{jk})\circ(\alpha_{ij}, \eta_{ij})=(\alpha_{jk}, \eta_{jk})$. This can be illustrated by the following commutative diagrams of $\CC$-morphisms and natural transformations.
\[
\xymatrix@!R@!C@R=20pt@C=20pt{
S_i\ar[r]^{\alpha_{ik}}\ar[d]_{\alpha_{ij}}&S_k   
&&G_i\ar@{=>}[r]^{\eta_{ik}}\ar@{=>}[d]_{\eta_{ij}}&G_k\circ\alpha_{ik}\\
S_j\ar[ur]_{\alpha_{ik}}&             
&&G_j\circ\alpha_{ij}\ar@{=>}[ur]_{(\eta_{jk})_{\alpha_{ij}}}
}
\]
E.g. in the particular case of inductive sequences, we have 
\[
\left\{
\begin{array}{ll}
\alpha_{ik}=\alpha_{k-1k}\circ\dots\circ\alpha_{ii+1}.\\
\eta_{ik}=(\eta_{k-1k})_{\alpha_{ik-1}}\circ\dots\circ(\eta_{i+1i+2})_{\alpha_{ii+1}}\circ\eta_{ii+1}.
\end{array}
\right.
\]
We now observe that the inductive system in $\CG$ induces an inductive system $(S_i, \alpha_{ij})_{i\in I}$ in $\CC$ and we denote $(S_\infty,\alpha_{i\infty})_i$ its inductive limit in $\CC$. The candidate for the $\CG$-limit will naturally be based upon $(S_\infty,\alpha_{i\infty})_i$ and we are next building a \emph{limit functor} $G_\infty:S_\infty\longrightarrow \AbGp$ together with \emph{limit transformations} $\{\eta_{i\infty}:G_i\Rightarrow G_\infty\circ\alpha_{i\infty}\}_{i\in I}$.\\
\textbf{[Construction of the limit functor $G_\infty$]} As a first step, we build an algebraic limit functor $G_\infty^{\alg}:\bigcup\limits_{i\in I}\alpha_{i\infty}(S_i)\longrightarrow \AbGp$. Let $s\in S_i$ and consider the inductive system in the category $\AbGp$, indexed by $J_i:=\{j\in I \mid j\geq i\}$
\[
\xymatrix{
G_i(s)\ar[r]^{(\eta_{ij})_s} &G(\alpha_{ij}(s))\ar[rr]^{(\eta_{jk})_{\alpha_{ij}(s)}} && G_k(\alpha_{ik}(s))\ar[r]^{}&\dots  & 
}
\]
We denote its limit by $(G_{\alpha_{i\infty}(s)},\eta_{\alpha_{ij}(s)})_{j\in J_i}$. We now have to distinguish the case where $\CC$ is either $\PoS$ or $\PoM$ and the case where $\CC$ is $\Cu$.

$\bullet$ $\CC$ is $\PoS$ or $\PoM$: Let $s\in S_i, t\in S_j$ such that $\alpha_{i\infty}(s)\leq \alpha_{j\infty}(t)$ in $S_{\infty}$. Then there exists some $k\geq i,j$ such that $\alpha_{ik}(s)\leq \alpha_{jk}(t)$. Therefore, we have the following commutative diagram in $\AbGp$
\[
\xymatrix{
G_i(s)\ar@/^{2pc}/[rrrrrr]_{\eta_s}\ar[rr]_{(\eta_{ik})_s} && G_k(\alpha_{ik}(s))\ar[d]_{G_k(\leq)}\ar[rr]_{(\eta_{kl})_{\alpha_{ik}(s)}} &&G_l(\alpha_{il}(s))\ar[d]_{G_l(\leq)}\ar[r]^{}& \dots\ar[d]\ar[r]^{}& G_{\alpha_{i\infty}(s)}\ar@{.>}[d]_{\exists!}   \\
G_j(t)\ar@/_{2pc}/[rrrrrr]^{\eta_t}\ar[rr]^{(\eta_{jk})_t} && G_k(\alpha_{jk}(t))\ar[rr]^{(\eta_{kl})_{\alpha_{jk}(t)}} &&G_l(\alpha_{jl}(t))\ar[r]^{}& \dots\ar[r]^{}& G_{\alpha_{j\infty}(t)}
}
\]
By the universal properties of the inductive limit, there exists a unique group morphism \break$G_{\alpha_{i\infty}(s)\alpha_{j\infty}(t)}:G_{\alpha_{i\infty}(s)}\longrightarrow G_{\alpha_{j\infty}(t)}$ that commutes with the above diagram. In particular, $G_{\alpha_{i\infty}(s)}\simeq G_{\alpha_{j\infty}(t)} $ whenever $\alpha_{i\infty}(s)=\alpha_{j\infty}(t)$.

$\bullet$ $\CC$ is $\Cu$: Let $s',s\in S_i, t\in S_j$ such that $s'\ll s$ and $\alpha_{i\infty}(s)\leq \alpha_{j\infty}(t)$ in $S_{\infty}$. Then there exists some $k\geq i,j$ such that $\alpha_{ik}(s')\leq \alpha_{jk}(t)$. Arguing similarly as in the previous case, there exists a unique group morphism $G_{\alpha_{i\infty}(s')\alpha_{j\infty}(t)}:G_{\alpha_{i\infty}(s')}\longrightarrow G_{\alpha_{j\infty}(t)}$. Now, observing that the set $s_{\ll}:=\{s'\ll s\}$ is upward directed and that the inductive system $(G(s'),G(s''\leq s'))_{s_{\ll}}$ has inductive limit $(G(s),G(s'\leq s)_{s_{\ll}}$, we conclude that there exists a unique group morphism $G_{\alpha_{i\infty}(s)\alpha_{j\infty}(t)}:G_{\alpha_{i\infty}(s)}\longrightarrow G_{\alpha_{j\infty}(t)}$.

Consequently, for all categories $\PoS,\PoM$ and $\Cu$, we have a well-defined functor
\[
\begin{array}{ll}
	G_\infty^{\alg}: \bigcup\limits_{i\in I}\alpha_{i\infty}(S_i)\longrightarrow \AbGp \\
	\hspace{1,65cm} \alpha_{i\infty}(s)\longmapsto G_{\alpha_{i\infty}(s)}\\
	\hspace{0,15cm}\alpha_{i\infty}(s)\leq \alpha_{j\infty}(t)\longmapsto G_{\alpha_{i\infty}(s)\alpha_{j\infty}(t)}
\end{array}
\] 
Whenever $\CC$ is the category $\PoS$ or the category $\PoM$, the inductive limits are algebraic. This is equivalent to saying that $\bigcup\limits_{i\in I}\alpha_{i\infty}(S_i)= S_\infty$. As a result, $(S_\infty,G_\infty^{\alg})$ belongs to $\CG$. 
Nevertheless, whenever $\CC$ is the category $\Cu$, extra-work is again needed. This is due to the fact that $\Cu$-semigroups carry a topological structure  (e.g. (O1) is a version of completeness). As a consequence, $\bigcup\limits_{i\in I}\alpha_{i\infty}(S_i)$ is dense in $S$, but no longer equal. Therefore, the algebraic limit functor has to be \textquoteleft completed\textquoteright.
As a second step, we build a completed limit functor $G_\infty:S_\infty\longrightarrow \AbGp$. Let $s\in S_\infty$ and let $(\alpha_{i_n\infty}(s_n))_n$ be a $\ll$-increasing sequence in $\bigcup\limits_{i\in I}\alpha_{i\infty}(S_i)$ whose supremum is $s$. Now consider the following inductive sequence in $\AbGp$
\[
\xymatrix{
\dots\ar[r]^{} &G_\infty^{\alg}(\alpha_{i_n\infty}(s_n))\ar[rr]^{G_\infty^{\alg}(\leq)} && G_\infty^{\alg}(\alpha_{i_{n+1}\infty}(s_{n+1}))\ar[r]^{}&\dots  & 
}
\]
We denote its limit by $(G_s,{\eta_{\alpha_{i_n\infty}(s_n)}})_{n\in\N}$. 
Let $s,t\in S_\infty$ be such that $s\leq t$ and pick $\ll$-increasing sequences $(\alpha_{i_n\infty}(s_n))_n,(\alpha_{j_n\infty}(t_n))_n$ in $\bigcup\limits_{i\in I}\alpha_{i\infty}(S_i)$ whose suprema are respectively $s,t$. Then for any $n$ there exists $m_n\geq n$ such that $\alpha_{i_n\infty}(s_n)\leq \alpha_{j_m\infty}(t_{m_n})$. Therefore, we have the following commutative diagram in $\AbGp$
\[
\xymatrix{
\dots\ar[r]^{}  &  
G_\infty^{\alg}(\alpha_{i_n\infty}(s_n))\ar[d]_{G_\infty^{\alg}(\leq)}\ar[rr]^{G_\infty^{\alg}(\leq)}  &&  
G_\infty^{\alg}(\alpha_{i_{n+1}\infty}(s_{n+1}))\ar[d]_{G_\infty^{\alg}(\leq)}\ar[r]^{}  &  
\dots\ar[d]\ar[r]^{}  &  
G_s\ar@{.>}[d]_{\exists!}   \\
\dots\ar[r]^{}  & 
G_\infty^{\alg}(\alpha_{j_{m_n}\infty}(t_{m_n}))\ar[rr]^{G_\infty^{\alg}(\leq)}   &&  
G_\infty^{\alg}(\alpha_{j_{m_{n+1}}\infty}(t_{m_{n+1}}))\ar[r]^{}  &
\dots\ar[r]^{}  &  
G_t
}
\]
By the universal properties of the inductive limit, there exists a unique group morphism $G_{st}:G_s\longrightarrow G_t$ that commutes with the above diagram. In particular, $G_s\overset{\id}\simeq G_t $ whenever $s=t$.
Consequently, we have a well-defined functor
\[
\begin{array}{ll}
	G_\infty: S_\infty\longrightarrow \AbGp \\
	\hspace{1,25cm} s\longmapsto G_s\\
	\hspace{0,7cm}s\leq t\longmapsto G_{st}
\end{array}
\] 
\textbf{[Construction of the limit transformations $\eta_{i\infty}$]} Let $i\in I$. We aim to build a natural transformation $\eta_{i\infty}:G_i\Rightarrow G_\infty\circ\alpha_{i\infty}$. Let $s,t\in S_i$ be such that $s\leq t$. We know from the construction of the algebraic limit functor that the following square is commutative.
\[
\xymatrix{
G_i(s)\ar[rr]^{\eta_s}\ar[d]_{G_i(\leq)} &&G_\infty^{\alg}(\alpha_{i\infty}(s))\ar[d]^{G_\infty^{\alg}(\leq)}   \\
G_i(t)\ar[rr]_{\eta_t} && G_\infty^{\alg}(\alpha_{i\infty}(t))
}
\]
Also, one can check that the restriction of $G_\infty$ to $\bigcup\limits_{i\in I}\alpha_{i\infty}(S_i)$ agrees with $G_\infty^{\alg}$.
Consequently, we have a well-defined natural transformation $\eta_{i\infty}:G_i\Rightarrow G_\infty\circ\alpha_{i\infty}$ where $(\eta_{i\infty})_s:=\eta_s$. We take the opportunity to remark that the following diagram of natural transformations commutes. 
\[
\xymatrix{
G_i\ar@{=>}[r]^{\eta_{i\infty}}\ar@{=>}[d]_{\eta_{ij}}&G_\infty\circ\alpha_{i\infty}\\
G_j\circ\alpha_{ij}\ar@{=>}[ur]_{(\eta_{j\infty})_{\alpha_{ij}}}
}
\]
We conclude that $(S_\infty,G_\infty)$ is a well-defined object in $\CG$ and  $(\alpha_{i\infty},\eta_{i\infty})$ is a well-defined morphism in $\CG$ for any $i\in I$. Further, gathering all the above, we deduce that 
\[
( (S_\infty,G_\infty), (\alpha_{i\infty},\eta_{i\infty}) )_{i\in I} \text{ is a cocone for } ( (S_i,G_i), (\alpha_{ij}, \eta_{ij}) )_{i\in I}.
\]

\begin{thm}
\label{thm:limG}
The categories $\PoSG,\PoMG$ and $\CuG$ have inductive limits. More particularly, we have that
\[
\CC_{*{\Gp}}-\lim\limits_{\longrightarrow}((S_i,G_i), (\alpha_{ij}, \eta_{ij}))_i\simeq ((S_\infty,G_\infty),(\alpha_{i\infty},\eta_{i\infty}))_i
\]
where $(S_\infty,\alpha_{i\infty})\simeq \CC-\lim\limits_{\longrightarrow}(S_i,\alpha_{ij})_i$, $G_\infty:S_\infty\longrightarrow \AbGp$ is the limit functor and  $\eta_{i\infty}:G_i\Rightarrow G_\infty\circ\alpha_{i\infty}$ are the limit transformations.
\end{thm}

\begin{proof}
From the construction above, we have already seen that $((S_\infty,G_\infty),(\alpha_{i\infty},\eta_{i\infty}))_i$ is a cocone for the inductive system. We are only left to check that it satisfies the universal property. Let $((T_\infty,H_\infty),(\beta_{i\infty},\nu_{i\infty}))_i$ be another admissible cocone for the inductive system. We have to show that there exists a unique $\CG$-morphism $(S_\infty,G_\infty)\longrightarrow (T_\infty,H_\infty)$ compatible with the cocone diagrams. Since $(S_\infty,\alpha_{i\infty})$ is the $\CC$-limit of $(S_i,\alpha_{ij})_i$, we know that there is a unique $\CC$-morphism $\gamma:S_\infty\longrightarrow T_\infty$ compatible with the cocones. We now have to find (unique) compatible group morphisms $G_\infty(s)\longrightarrow H_\infty(\gamma(s))$ that will lead to a (unique) compatible natural transformation $G_\infty\Rightarrow H_\infty\circ\gamma$. 
Let $s\in S_i$. Observe that the cocones in $\CG$ induce the following cocones in $\AbGp$.
\[
\xymatrix{
G_i(s)\ar[r]^{(\eta_{ij})_s} &G(\alpha_{ij}(s))\ar[r]^{(\eta_{jk})_{\alpha_{ij}(s)}} & 
G_k(\alpha_{ik}(s))\ar@/_{2pc}/[rr]^{(\eta_{k\infty})_{\alpha_{ik}(s)}}\ar@/_{4pc}/[rrr]^{(\nu_{k\infty})_{\beta_{ik}(s)}}\ar[r]^{}&
\dots\ar[r]^{}  & 
G_\infty(\alpha_{i\infty}(s))\ar@{.>}[r]^{\exists!} & H_\infty(\beta_{i\infty}(s))\\
&&&&&\\
&
}
\]
However, we have constructed $(G_\infty(\alpha_{i\infty}(s)),(\eta_{k\infty})_{\alpha_{ik}(s)})_k$ to be the inductive limit of the above system. Thus there exists a unique compatible group morphism $\mu_{\alpha_{i\infty}(s)}:G_\infty(\alpha_{i\infty}(s)) \longrightarrow H_\infty(\beta_{i\infty}(s))$.  Arguing similarly as in the construction of the limit functor and noticing that $\gamma\circ\alpha_{i\infty}=\beta_{i\infty}$, we get that for $s,t\in S_i$ such that $s\leq t$, the following square is commutative.
\[
\xymatrix@!R@!C@R=40pt@C=60pt{
G_i(s)\ar@/^{-3pc}/[dd]_{({\nu_{i\infty}})_s}\ar[r]^{G_i(\leq)}\ar[d]^{({\eta_{i\infty}})_s} & 
G_i(t)\ar@/_{-3pc}/[dd]^{({\nu_{i\infty}})_t}\ar[d]_{({\eta_{i\infty}})_t} \\
G_\infty(\alpha_{i\infty}(s))\ar[r]^{G_\infty(\leq)}\ar[d]^{\mu_{\alpha_{i\infty}(s)}} & 
G_\infty(\alpha_{i\infty}(t))\ar[d]_{\mu_{\alpha_{i\infty}(t)}}\\
H_\infty\circ\gamma(\alpha_{i\infty}(s))\ar[r]_{H_\infty(\leq)}& 
H_\infty\circ\gamma(\alpha_{i\infty}(t))
}
\]

Overall, we have constructed a unique natural transformation $\mu:G_\infty^{\alg}\Rightarrow H_\infty\circ \gamma$ compatible with the cocones. Again arguing similarly as in the construction of the limit transformations, it can be shown that $\mu$ can be extended to a (unique) compatible natural transformation $\mu:G_\infty\Rightarrow H_\infty\circ \gamma$, in the sense that $\nu_{i\infty}=\mu_{\alpha_{i\infty}}\circ \eta_{i\infty}$. 
\end{proof}

$\hspace{-0,34cm}\bullet\,\,\textbf{Inductive limits in the categories }\PoS,\oMs \textbf{\and} \Cus$. The next step towards the continuity of the webbing functors is to prove that the codomain categories also have inductive limits. Thankfully, this task is simpler as we already know that $\PoS$ has inductive limits and that $\oMs$ and $\Cus$ are full subcategories of $\oM$ and $\Cu^\sim$ respectively, which both have inductive limits. Therefore, we are only left to check that axioms (PC), (PD) and (S0) pass to inductive limit in these larger categories. We next focus on the case of the category $\Cus\subseteq \Cu^\sim$ and a similar proof is valid for the category $\oMs$. Let us first recall a characterization of inductive limits in the category $\Cu^\sim$ given in \cite[Corollary 4.8]{AL23}.

Let $(S_i,\sigma_{ij})_i$ be an inductive system in $\Cu^\sim$. Then a pair $(S,\sigma_{ij})_i$ is the $\Cu^\sim$-inductive limit of the system if and only if it is a cocone satisfying the following axioms:

(L1) For any $s',s\in S$ such that $s'\ll s$, there exists $s_i$ in some $S_i$ such that $s'\ll \sigma_{i\infty}(s_i)\ll s$.

(L2) For any $s,t\in S_i$ such that $\sigma_{i\infty}(s)\leq \sigma_{i\infty}(t)$ and any $s'\in S_i$ such that $s'\ll s$, then $\sigma_{ij}(s')\ll \sigma_{ij}(t)$ for some $j\geq i$.

We mention that a sequential version of the above characterization in the spirit of \cite[Section 2.1]{R12} can be deduced.
\begin{prop}
\label{prop:Cuslim}
Let $(S_i,\sigma_{ij})_i$ be an inductive system in $\Cus$ and let $(S,\sigma_{i\infty})_i$ be its inductive limit in $\Cu^\sim$. Then $S$ is a $\Cus$-semigroup and $\sigma_{i\infty}$ is a $\Cus$-morphism for any $i\in I$.

As a result, the category $\Cus$ has inductive limits, which are also characterized by (L1) - (L2).
\end{prop}

\begin{proof} The strategy consists in showing that $S$ satisfies axioms (PC), (PD), (S0), in other words, that $S$ is a $\Cus$-semigroup. The rest of the proposition will follow from the fact that $\Cus$ is a full subcategory of $\Cu^\sim$.

(PC) Let $s\in S$ and let $t\in S_+$ such that $s\leq t$. Using (L1), we can find $s_i,t_i$ in some $S_i$ with $t_i\in (S_i)_+$ such that $\sigma_{i\infty}(s_i)\leq s, \sigma_{i\infty}(t_i)\leq t$ and $\sigma_{i\infty}(s_i)\ll \sigma_{i\infty}(t_i)$. Now, using (L2) we deduce that there exists $j\geq i$ such that $\sigma_{i\infty}(s_i)\ll \sigma_{i\infty}(t_i)$. It follows from axiom (PC) in $S_j$ that $\sigma_{i\infty}(s_i)+\sigma_{i\infty}(t_i)\geq 0$ and thus $s+t\geq 0$.

(PD) Let $s\in S$ and let $s_i\in S_i$ such that $\sigma_{i\infty}(s_i)\leq s$. (Such an element exists by   (L1).) It follows from axiom (PD) in $S_i$ that there exists $p\in S_i$ such that $s_i+p\geq 0$ and hence, $s+\sigma_{i\infty}(p)\geq \sigma_{i\infty}(s_i+p)\geq 0$.

(S0) Let $s,t\in S$ be such that $s+t=0$. Take two $\ll$-increasing sequences $(x_n)_n,(y_n)_n$ given by (O2) whose suprema are respectively $s,t$. Then for any $n\in\N$, one can find $s_n,t_n$ in some $S_{i_n}$ such that $x_n\ll\sigma_{i_n\infty}(s_n)\ll s$ and $y_n\ll\sigma_{i_n\infty}(t_n)\ll t$. Observe that $\sigma_{i_n\infty}(s_n)+\sigma_{i_n\infty}(t_n)\ll 0$ and we deduce from (L2) that $\sigma_{i_nj_n}(s_n)+\sigma_{i_nj_n}(t_n)\ll 0$ for some $j_n\geq i_n$. It follows from axioms (PC) and (S0) in $S_{j_n}$ that $\sigma_{i_nj_n}(s_n)=\sigma_{i_nj_n}(t_n)=0$. As a result, we obtain that $x_n,y_n\ll 0$ for any $n\in\N$. Lastly, using axiom (PC) in $S$ (that we have just proven to hold) we conclude that  $x_n=y_n=0_S$ for any $n\in\N$ and hence that $s=t=0$.
\end{proof}

As commented above, the arguments of the latter proof are valid in the (algebraic) case of ordered monoids and hence, the category $\oMs$ also have inductive limits. In the following theorem, $\CC$ stands for either $\PoS,\PoM$ or $\Cu$ and we conventionally write $\PoS^*$ to be the category $\PoS$.
\begin{thm}
\label{thm:webctn}
The webbing functors $\mathfrak{w}:\CG\longrightarrow \CC^*$ are continuous with respect to inductive limits.
\end{thm}

\begin{proof}
Let $( (S_i,G_i), (\alpha_{ij}, \eta_{ij}) )_{i\in I}$ be an inductive system in $\CC_{*{\Gp}}$ and let $((S_\infty,G_\infty),(\alpha_{i\infty},\eta_{i\infty}))_i$ be its inductive limit. The webbing functor induces an inductive system  in $\CC^*$ and let $(S,\sigma_{i\infty})_i$ its inductive limit in $\CC^*$. 
For notational purposes, we write ${(S_G)}_{i}:= \mathfrak{w}(S_i,G_i)$ instead of ${(S_i)}_{G_i}$ and ${(\alpha_\eta)}_{ij} :=\mathfrak{w}(\alpha_{ij},\eta_{ij})$ instead of $(\alpha_{ij})_{\eta_{ij}}$. Similarly for the limit. 

The webbing functor also induces a cocone $( {(S_G)}_{\infty}, ({(\alpha_\eta)}_{ij} )_i$ for the latter inductive system in $\CC^*$. By the universal property, there exists a unique $\CC^*$- morphism $\gamma:S\longrightarrow ({S_\infty})_{G_\infty}$ such that the following diagram commutes
\[
\xymatrix{
{(S_G)}_{i}\ar[dd]_{{(\alpha_\eta)}_{ij}}\ar[rd]^{\sigma_{i\infty}}\ar@/^2pc/[rrd]^{{(\alpha_\eta)}_{i\infty}} &  &\\
& S\ar@{.>}[r]_{\hspace{-0,9cm}\exists!\, \gamma}&  {(S_G)}_{\infty}\\
{(S_G)}_{j}\ar[ru]_{\sigma_{j\infty}}\ar@/_2pc/[rru]_{{(\alpha_\eta)}_{j\infty}}  & &
} 
\]
To complete the proof, let us show that $\gamma$ is a surjective order-embedding. We establish the property for the case $\CC=\Cu$ and the other (easier) cases are proven similarly. We start by stating a fact that we use several times in the sequel and that is obtained after combining (O2) and the characterization (L1) - (L2).

\textit{Fact}: Any element in $S$ can be approximated by an ultimately $\ll$-increasing sequence $(x_n)_n$ where $x_n:=\sigma_{i_n\infty}(s_n,g_n)$ and such that $(\alpha_\eta)_{i_ni_{n+1}}(	s_n,g_n)\ll (s_{n+1},g_{n+1})$ in ${(S_G)}_{i_{n+1}}$. 

Injectivity: Let $s,t\in S$ be such that $\gamma(s)\leq\gamma(t)$. Consider $(x_n)_n$ and $(y_n)_n$ as in the claim, where $x_n=\sigma_{i_n\infty}(s_n,g_n)$ and $y_n=\sigma_{j_n\infty}(t_n,h_n)$. Let $n\in\N$. We know that $(\gamma(y_n))_n$ is $\ll$-increasing towards $\gamma(t)$ and $\gamma(x_n)\ll \gamma(t)$. Thus, there exists some $m\in\N$ such that $\gamma(x_n)\ll \gamma(y_m)$. By the commutativity of the above diagram, this is equivalent to saying 
\[
\left\{
\begin{array}{ll}
\alpha_{i_n\infty}(s_n)\ll \alpha_{j_m\infty}(t_m)\\
\,[G_\infty(\alpha_{i_n\infty}(s_n)\leq \alpha_{j_m\infty}(t_m))\circ(\eta_{i_n\infty})_{s_n}](g_n)= (\eta_{j_m\infty})_{t_n}(h_n)
\end{array}
\right.
\]
Using (L2), we can find some $k\in I$ big enough such that $\alpha_{i_nk}(s_n)\ll \alpha_{j_mk}(t_m)$. Further, we know that the following inductive system is commutative and that the limits are algebraic. 
\[
\xymatrix{
G_{i_n}(s_n)\ar[rr]^{(\eta_{i_nk})_{s_n}} && G_k(\alpha_{i_nk}(s_n))\ar[d]^{G_k(\leq)}\ar[r]^{} &\dots\ar[r]^{}  & G_\infty(\alpha_{i_nk}(s_n))\ar[d]^{G_\infty(\leq)} \\
G_{j_m}(t_m)\ar[rr]^{(\eta_{j_mk})_{t_m}} && G_k(\alpha_{j_mk}(t_m))\ar[r]^{} &\dots\ar[r]^{}  & G_\infty(\alpha_{i_nk}(s_n))
}
\]
Therefore, we can find some $l\in I$ big enough such that \[[(\eta_{kl})_{\alpha_{j_mk}(t_m)}\circ G_k(\alpha_{i_nk}(s_n)\leq \alpha_{j_mk}(t_m))\circ(\eta_{i_nk})_{s_n}](g_n)= (\eta_{j_ml})_{t_n}(h_n)
\]
Putting everything together, we have proven that $({\alpha_\eta})_{i_nl}(x_n)\leq ({\alpha_\eta})_{j_ml}(y_m)$. Applying $\sigma_{l\infty}$ on both sides, we get that $\sigma_{i_n\infty}(x_n)\leq \sigma_{j_m\infty}(y_m)$ in $S$. Passing to suprema first right, then left we finally conclude that $x\leq y$ and hence, that $\gamma$ is an order-embedding. 

Surjectivity: Let $(s_\infty,g_\infty)\in {(S_G)}_\infty$. Assume in a first instance that $s_\infty\in \bigcup\limits_{i\in I}\alpha_{i\infty}(S_i)$. That is, $s_\infty=\alpha_{i\infty}(s)$ for some $s\in S_i$.    A fortiori, $g_\infty$ belongs $G_\infty(\alpha_{i\infty}(s_i))$ which is the (algebraic) inductive limit of 
\[
\xymatrix{
G_i(s)\ar[r]^{(\eta_{ij})_s} &G(\alpha_{ij}(s))\ar[rr]^{(\eta_{jk})_{\alpha_{ij}(s)}} && G_k(\alpha_{ik}(s))\ar[r]^{}&\dots  & 
}
\]
Therefore, we can find some $g\in G_j(\alpha_{ij}(s))$ such that $(\eta_{j\infty})_{\alpha_{ij}(s)}(g)=g_\infty$. Overall, we have find an element $(\alpha_{ij}(s),g)\in {(S_G)}_j$ such that $({\alpha_\eta})_{j\infty}(\alpha_{ij}(s),g)=(s_\infty,g_\infty)$. By commutativity of the former diagram, this is equivalent to saying that $\gamma\circ\sigma_{j\infty}(\alpha_{ij}(s),g)=(s_\infty, g_\infty)$. We conclude that any element in $\bigcup\limits_{i\in I}{(\alpha_\eta)}_{i\infty}(({S_G})_i)$ has an antecedent. Next, let us deal with the general case. Let $(s_\infty,g_\infty)\in {(S_G)}_\infty$. It can be easily verified that $\bigcup\limits_{i\in I}{(\alpha_\eta)}_{i\infty}(({S_G})_i)$ is dense in ${(S_G)}_\infty$. In particular, there exists an increasing sequence $(x_n)_n$ in $\bigcup\limits_{i\in I}{(\alpha_\eta)}_{i\infty}(({S_G})_i)$ whose supremum is $(s_\infty,g_\infty)$. Applying what we have just proved, we can find $(s_n,g_n)\in {(S_G)}_{i_n}$ such that $x_n={(\alpha_\eta)}_{i_n\infty}(s_n,g_n)$ for any $n\in\N$. (We can even arrange that $i_n\leq i_{n+1}$.) Let us write $z_n:=\sigma_{i_n\infty}(s_n,g_n)$. The sequence $(z_n)_n$ is an increasing sequence in $S$ and hence, its supremum $z$ exists. Lastly, by commutativity of the former diagram, we easily conclude that $\gamma(z)=(s_\infty,g_\infty)$ which proves that $\gamma$ is a $\Cus$-isomorphism.
\end{proof}

\section{Ideals and quotients in the category \texorpdfstring{$\Cus$}{Cu*}}
This section is dedicated to the study of the ideal structure of abstract $\Cus$-semigroups together with remarkable substructures such as the sets of positive and maximal elements. We are also interested in quotients by ideals and exact sequences that link all the above together.  A theory of ideals and quotients has already been developed in \cite{C21b} for $\Cu^\sim$-semigroups satisfying axioms (PC) and (PD), modeled after the one for $\Cu$-semigroups developed in \cite{CRS10}. Recall that $\Cus$ is the full subcategory of $\Cu^\sim$ whose objects are satisfying axioms (PC), (PD) and (S0). Therefore, in the sequel we overall transpose most of these results to abstract $\Cus$-semigroups and we will provide extra-details when needed.

\vspace{0,2cm}$\hspace{-0,34cm}\bullet\,\,\textbf{The positive cone and the maximal elements}$.
Let $S$ be an abstract $\Cus$-semigroup. We naturally define its positive cone $S_+$ and the set of its maximal elements $S_{\max}$ by 
\[
\begin{array}{ll}
S_+:=\{s\in S \mid  0\leq s\}\\
S_{\max}:=\{s\in S \mid \text{ if } t\geq s, \text{ then } t=s\}
\end{array}
\]
While it is obvious that $S_+$ is a $\Cu$-semigroup, it is not clear whether $S_{\max}$ has a remarkable structure. We expose is the next proposition, which falls from \cite[Proposition 5.4]{C21a} and \cite[Lemma 3.8]{C21b}, that $S_{\max}$ is either empty or has an abelian group structure.  \\

\begin{prop}
Let $S$ be a $\Cus$-semigroup. The following are equivalent

(i) $S_{\max}$ is not empty, i.e. $S$ has maximal elements.

(ii) $S_+$ has a largest element.

(iii) $S_{\max}$ is an absorbing abelian group.\\
Moreover, if one of the above is true, then

(iv) The neutral element of $S_{\max}$ is the largest element of $S_+$.

(v) For any $x\in S$, there exists a unique $p_x\in S_{\max}$ such that $x+p_x=e_{S_{\max}}\geq 0$.
\end{prop}

Recall that an abstract $\Cus$-semigroup $S$ is said to be \emph{countably-based} whenever there exists a countable subset $B$ such that, for any $s,t\in S$ with $s\ll t$ there exists $b\in B$ with $s\leq b\ll t$. Equivalently, any element $s\in S$ can be written as the supremum of a $\ll$-increasing sequence of elements of $B$.

It is a well-known fact that any countably-based $\Cu$-semigroup has a largest element. In particular, any countably-based $\Cus$-semigroup $S$ has a countably-based positive cone and hence, has maximal elements. ($S_{\max}$ contains at least a neutral element, which happens to be the largest element of $S_+$.) 
In the sequel, whenever mentioning or using $S_{\max}$, we shall explicitly recall (as much as possible) or implicitly assume that $S$ has maximal elements. In particular, the functor $\nu_{max}$ of the next theorem is formally defined from the subcategory of $\Cus$ whose objects have maximal elements, that we also write $\Cus$. 
\begin{thm}
There exist continuous functors (with respect to inductive limits)
\[
\begin{array}{ll}
	\nu_{+}: \Cus \longrightarrow \Cu		
	\hspace{3,3cm}\nu_{\max}: \Cus \longrightarrow \AbGp\\
		\hspace{1,1cm}S \longmapsto S_{+}
		\hspace{4,75cm} S \longmapsto S_{\max}\\
		\hspace{1,15cm}\alpha \longmapsto \alpha_{+}
		\hspace{4,7cm} \alpha \longmapsto \alpha_{\max}
\end{array}
\]
where $\alpha_+:S_+\longrightarrow T_+$ and $\alpha_{\max}:S_{\max}\longrightarrow T_{\max}$ are obtained by restrictions of $\alpha:S\longrightarrow T$ to their respective (co)domains and defined by $\alpha_{+}:=\alpha$ and $\alpha_{\max}:=\alpha + e_{T_{\max}}$.
\end{thm}

\begin{proof}
The fact that the above assignments are well-defined functors between their respective categories is almost immediate and done e.g. in \cite[Lemma 3.9 - Proposition 5.5]{C21a}. Here, we prove the continuity of these functors. Let $(S_i,\sigma_{ij})_i$ be an inductive system in $\Cus$ and let $(S,\sigma_{i\infty})_i$ be its inductive limit in $\Cus$. Abusing notations, observe that $\nu_+((S,\sigma_{i\infty})_i)$ and $\nu_{\max}((S,\sigma_{i\infty})_i)$ are respective cocones for the respective induced systems $\nu_+((S_i,\sigma_{ij})_i)$ and $\nu_{\max}((S_i,\sigma_{ij})_i)$. Write $(T,\alpha_{i\infty})_i$ and $(G,\beta_{i\infty})_i$ the respective inductive limits of the latter inductive systems in their respective categories. More specifically $(T,\alpha_{i\infty})_i$ is the $\Cu$-limit of $\nu_+((S_i,\sigma_{ij})_i)$ and $(G,\beta_{i\infty})_i$ is the $\AbGp$-limit of $\nu_{\max}((S_i,\sigma_{ij})_i)$. By universal property of the inductive limit, we obtain a $\Cu$-morphism $\alpha:T\longrightarrow S_+$ and a $\AbGp$-morphism $\beta:G\longrightarrow S_{\max}$ compatible with the cocones $\nu_+((S,\sigma_{i\infty})_i)$ and $\nu_{\max}((S,\sigma_{i\infty})_i)$. We are left to show an injective-surjective type of argument. This is done similarly as in the proof of \autoref{thm:webctn}. In fact, the very same arguments are easily applied in this simpler context and we left the reader to check them.
\end{proof}

We next expose a split-exact sequence in $\Cus$ that involves a $\Cus$-semigroup with maximal elements, its positive cone $S_+$ and its maximal elements $S_{\max}$. Observe that the category $\Cus$ is not an abelian category and one has to be cautious when mentioning kernels, cokernels and exactness of a sequence. We recall some facts from \cite[\S 4.6]{C21b}.

\begin{dfn}
Let $S,T$ be $\Cus$-semigroups and let $\alpha:S\longrightarrow T$ be $\Cus$-morphism. We define 
\[
\begin{array}{ll}
\im \alpha:= \{ (t_1,t_2)\in T\times T \mid \text{there exists } s\in S \text{ with } t_1\leq \alpha(s)+t_2 \}\\
\ker \beta:= \{ (s_1,s_2)\in S \times S\mid \alpha(s_1)\leq \alpha(s_2) \}
\end{array}
\]
A sequence $\hspace{0,3cm}  \dots\,\longrightarrow S\overset{\alpha}\longrightarrow T\overset{\beta}\longrightarrow V\longrightarrow\,\dots \hspace{0,3cm}$ in the category $\Cus$ is \emph{exact at $T$} if $\ker \beta= \im \alpha$.  A short sequence $\hspace{0,3cm} 0\longrightarrow S\overset{\alpha}\longrightarrow T\overset{\beta}\longrightarrow V\longrightarrow 0\hspace{0,3cm}$ in the category $\Cus$ is \emph{exact} if it is exact everywhere. If moreover there exists a $\Cus$-morphism $\rho:V\longrightarrow T$ such that $\beta\circ\rho=\id_V$, we say that the short sequence is \emph{split-exact}.
\end{dfn}

These notions of kernel and exactness allow to rewrite order-embeddings and exhaustive maps via exact sequences, as it is usually done in abelian categories. Namely, we have

\begin{prop}\label{prop:seqinjsur} \cite[Proposition 4.8]{C21b} Let $S,T$ be $\Cus$-semigroups and let $\alpha:S\longrightarrow T$ be $\Cus$-morphism. Then

(i) $\alpha$ is an order-embedding if and only if $0\longrightarrow S\overset{\alpha}\longrightarrow T$ is exact.

(ii) $\alpha$ is surjective if and only if $ S\overset{\alpha}\longrightarrow T\longrightarrow 0$ is exact and $\alpha(S)$ is order-hereditary in $T$.
\end{prop}

\begin{thm}\label{thm:splitmax}\cite[Theorem 4.13 - Corollary 4.14]{C21b}
Let $S$ be a $\Cus$-semigroup with maximal elements. The following short sequence is split-exact in the category $\Cus$
\[
\xymatrix{
0\ar[r]^{} & S_+\ar[r]^{i} & S\ar[r]^{j} & S_{max}\ar@/_{-1pc}/[l]^{q}\ar[r]^{} & 0
} 
\]
where $i$ and $q$ are the canonical inclusions and $j$ is defined by $j(s):=s+e_{S_{\max}}$.

Further, consider $\alpha:S\longrightarrow T$ is a $\Cus$-morphism between $\Cus$-semigroups with maximal elements. Then the following diagram is commutative with exact rows in the category $\Cus$
\[
\xymatrix@!R@!C@R=30pt@C=30pt{
0\ar[r]^{} & S_+\ar[d]_{\alpha_+}\ar[r]^{i_S} & S \ar[d] _{\alpha}\ar[r]^{j_S} & S_{\max}\ar@/_{-1pc}/[l]^{q_S} \ar[d]^{\alpha_{\max}}\ar[r]^{} & 0
\\
0\ar[r]^{} & T_+\ar[r]^{i_T} & T\ar[r]^{j_T} & T_{\max}\ar@/_{-1pc}/[l]^{q_T}\ar[r]^{} & 0
} 
\]
\end{thm}

Using the above diagram, it can be shown that $\alpha_+$ and $\alpha_{\max}$ are isomorphisms in their respective categories whenever $\alpha$ is a $\Cus$-isomorphism. Nevertheless, the converse is not true as shows the example constructed in \cite{C23}. More particularly, we are missing the entire ideal structure that we dive into next.

\vspace{0,2cm}$\hspace{-0,34cm}\bullet\,\,\textbf{$\Cus$-Ideals}$. Knowing about positive and maximal elements (when they exist) of a $\Cus$-semigroup $S$ allows us to exploit the ideal structure within $S$. Again most definitions and results are an adaptation of \cite{C21b} and we will provide extra-details or proofs whenever they are needed. 

\begin{dfn}
Let $S$ be a $\Cus$-semigroup. An \emph{ideal} of $S$ is a subset $I$ such that 

(i) $I$ is a positively directed submonoid of $S$.

(ii) $I$ is closed under suprema of increasing sequences and order-hereditary. (Equivalently, $I$ is a closed set for the Scott topology.)

(iii) For any $s,t\in S$ such that $s+t\in I$ then $s,t\in I$. (Positively Stable axiom)
\end{dfn}

We remark that the notion of \emph{positively stable} have been slightly modified. The notion exposed here is equivalent but somehow clearer. Similarly, in what follows the computation of the ideal generated by an element have been rewritten in an equivalent but neater way. Lastly, we mention that axiom (S0) ensures that for any $\Cus$-semigroup $S$, the set $\{0\}$ is always a $\Cus$-ideal of $S$ and in particular, $\{0\}$ is the ideal generated by $0_S$. (This is no longer true in the category $\Cu^\sim$ as shown in \cite[Section 5 - Proposition]{C23} where the author exhibits a $\Cu^\sim$-semigroup $S$ such that the ideal generated by $0_S$ is strictly larger than $\{0\}$.)

\begin{prop}\cite{C21b}
\label{prop:idealppty}
Let $S$ be a $\Cus$-semigroup and let $I$ be an ideal of $S$. Then

(i) $I$ is a $\Cus$-semigroup that canonically order-embeds into $S$. In particular, the sequence $0\longrightarrow I \longrightarrow S$ is exact.

(ii) For any $x\in  S_+$, the set $I_x:=\{y\in S\mid \text{there exists }z\in S \text{ with } y+z\leq \infty x\}$ is the smallest ideal of $S$ containing $x$ and we refer to it as the \emph{ideal generated by $x$}. 
Moreover, it is always true that $I_0={0_S}$. 

(ii') The smallest ideal containing an element might not exist. (See \cite[Proposition 3.12]{C21b}.)

(iii) $I$ has maximal elements if and only if $I$ is singly-generated (e.g. by its largest positive element $e_{I_{\max}}$.) 

(iv) There is a lattice isomorphism $\Lat(S)\simeq \Lat(S_+)$, where $\Lat(S)$ and $\Lat(S_+)$ denote the lattices of ideals of $S$ and $S_+$ respectively. Moreover, it matches $\Lat_f(S)\simeq\Lat_f(S_+)$, where $\Lat_f(S)$ and $\Lat_f(S_+)$ denote the sets of singly-generated ideals of $S$ and $S_+$ respectively.
\end{prop}

\begin{proof}
We give a brief argument to justify the \textquoteleft updated\textquoteright\ form of the set $I_x$. Originally, \cite[Proposition 3.12]{C21b} stipules that the smallest ideal generated by an element $x\in S_+$ is $\{y\in S\mid \text{there exists }z\in S \text{ with } 0\leq y+z\leq \infty x\}$. The latter set is obviously contained in $I_x$. However, for any $y\in I_x$ there exists $z\in S$ such that $y+z\leq \infty x$ and hence $y+z+\infty x\leq 2 \infty x=\infty x$. Further, by axiom (PC) we deduce that $y+z+\infty x\geq 0$. Writing $z':= z+\infty x$, it follows that $y\in \{y\in S\mid \text{there exists }z\in S \text{ with } 0\leq y+z\leq \infty x\} $ and the two sets agree.

We also give brief argument to (iv). For any $I\in \Lat(S_+)$, consider $I^*:=\{s\in S \mid \exists t\in S \text{ with } s+t\in I\}$. It can be checked that $I^*\in \Lat(S)$ and that $I\longmapsto I^*$ is a bijection whose inverse map sends an ideal of $S$ to its positive cone.
\end{proof}

We end up this part by showing that the ideal structure and the exact sequence linking the positive cone and maximal elements are all \textquoteleft compatible\textquoteright\ with $\Cus$-morphisms. \begin{thm}
\label{thm:diagrammorph}\cite[Corollaries 3.22 - 4.14 - 4.15]{C21b}
Let $S,T$ be $\Cus$-semigroups with maximal elements and let $\alpha:S\longrightarrow T$ be a $\Cus$-morphism. Let $I\in\Lat_f(S)$ and we write $I_\alpha$ to denote the ideal of $T$ generated by $\alpha(e_{I_{\max}})$. 

 Then the following diagram is commutative with exact rows in the category $\Cus$.\\
\[
\xymatrix@!R@!C@R=20pt@C=2pt{
&0\ar[rr] &&  (I_\alpha)_+\ar[rr]\ar@{.>}[dd] &&  I_\alpha\ar[rr]\ar@{.>}[dd] &&  (I_\alpha)_{\max}\ar@{.>}@/_{-1.2pc}/[ll]\ar[rr]\ar@{.>}[dd] && 0 \\
0\ar[rr] && I_+\ar[rr]\ar[dd]\ar[ur]_{(\alpha_{\mid })_+} && I\ar[rr]\ar[dd]\ar[ur]_{\alpha_{\mid }} && I_{\max}\ar@/_{-1.2pc}/[ll]\ar[rr]\ar[dd]\ar[ur]_{(\alpha_{\mid })_{\max}} && 0 \\
&0\ar@{.>}[rr] && T_+\ar@{.>}[rr] && T\ar@{.>}[rr] && T_{\max}\ar@{.>}@/_{-1.2pc}/[ll]\ar@{.>}[rr] && 0\\
0\ar[rr]{} && S_+\ar[rr]\ar@{.>}[ur]_{\alpha_+} && S\ar[rr]\ar@{.>}[ur]_{\alpha} && S_{\max}\ar@/_{-1.2pc}/[ll]\ar[rr]\ar@{.>}[ur]_{\alpha_{\max}} && 0
}
\vspace{1cm}\]
where $\alpha_{\mid }:I\longrightarrow I_\alpha$ is the restriction of $\alpha$, the horizontal arrows are the respective maps $i,j,q$ as in \autoref{thm:splitmax} and the vertical arrows are $\nu_+(I\overset{\subseteq}\longrightarrow S), I\overset{\subseteq}\longrightarrow S, \nu_{\max}(I\overset{\subseteq}\longrightarrow S)$ and similarly with $I_\alpha\overset{\subseteq}\longrightarrow T$. 
\end{thm}
\begin{rmk}
(i) Observe that $(\alpha_{\mid})_+(i_0)=\alpha(i_0)$ for all $i_0\in I_+$ and $(\alpha_{\mid})_{\max}(i_\infty)=\alpha(i_{\infty})+e_{(I_\alpha)_{\max}}=\alpha(i_{\infty})+\alpha(e_{I_{\max}})=\alpha(i_{\infty})$ for all $i_{\infty}\in I_{\max}$.

(ii) We have assumed that $S,T$ have maximal elements to ensure that $S_{\max},T_{\max}$ are not empty and $\alpha_{\max}$ exists. For the general case, an analogous result applies replacing $S,T$ by any singly-generated ideals $I',J$ that respectively contain $I,I_\alpha$ and such that $J\supseteq I_{\alpha}$. 
\end{rmk}

$\hspace{-0,34cm}\bullet\,\,\textbf{$\Cus$-Quotients}$. Let $S$ be a $\Cus$-semigroup and let $I\in \Lat(S)$. We can define a preorder on $S$ by \[x\leq_I y \text{ if there exists } z\in I \text{ such that } x\leq y+z.\] This induces an equivalence relation $\sim_I$ obtained by antisymmetrisation of $\leq_I$ and we define the \emph{$\Cus$-semigroup quotient of $S$ by $I$} as follows\[S/I:=(S/\!\sim_I,+,\leq)\] where \[[s]+[t]:=[s+t] \text{ and }[s]\leq [t] \text{ if } s\leq_I t\]

\begin{prop}\label{prop:idealquotient} \cite[Lemma 4.4]{C21b}
Let $S,T$ be $\Cus$-semigroups and let $\alpha:S\longrightarrow T$ be a $\Cus$-morphism.  Let $I\in\Lat(S)$. Then

(i) $S/I$ is a well-defined $\Cus$-semigroup.

(ii) The short sequence $0\longrightarrow I \longrightarrow S\longrightarrow S/I \longrightarrow 0$ is exact.
 
(iii) $\alpha$ factorizes through $S/I$ whenever $I\subseteq \alpha^{-1}(\{0_T\})$. In that case, the factorized map $\overline{\alpha}:S/I\longrightarrow T$ is surjective if and only if $\alpha$ is surjective. 
\end{prop}

\begin{proof}
(i) and (iii) are done in \cite{C21b}. We are left to prove (ii). Exactness at $I$ and $S/I$ are automatic from \autoref{prop:seqinjsur}. Now observe $(s_1,s_2)\in \ker \pi$ if and only if $[s_1]\leq [s_2]$ in $S/I$ if and only if there exists $z \in I$ such that $s_1\leq i(z)+s_2$ if and only if $(s_1,s_2)\in \im i$ which ends the proof.
\end{proof}

\begin{cor} 
Let $S,T$ be $\Cus$-semigroups and let $\alpha:S\longrightarrow T$ be a $\Cus$-morphism.  Let $I\in\Lat(S)$ and let $J\in\Lat(T)$ such that $J\supseteq \alpha(I)$. Then $\alpha$ induces a $\Cus$-morphism $\overline{\alpha}:S/I\longrightarrow T/J$ such that the following diagram commutes
\[
\xymatrix{
 S\ar[d]_{}\ar[r]^{\alpha} & T \ar[d] _{}
\\
S/I\ar[r]^{\overline{\alpha}} & T/J
} 
\] 
\end{cor}

\begin{proof}
This is a direct corollary from (iii) of the above proposition.
\end{proof}
Let us end this part by stating how the all above is applied for specific classes of $\Cus$-semigroups that we refer to as webbed $\Cus$-semigroups and concrete $\Cus$-semigroups.
\begin{prg}\textbf{Webbed $\Cus$-semigroups - Concrete $\Cus$-semigroups}.  A \emph{webbed $\Cus$-semigroup}
is a $\Cus$-semigroup of the form $S_G:=\mathfrak{w}(S,G)$ for some $(S,G)\in \CuG$. Similarly a \emph{webbed $\Cus$-morphism} between webbed $\Cus$-semigroups $S_G$ and $T_H$ is a $\Cus$-morphism of the form $\eta_\alpha:=\mathfrak{w}(\eta,\alpha)$ for some $(\eta,\alpha)\in \Hom_{\CuG}((S,G),(T,H))$. 
Let us briefly remark how the ideal-quotient theory can be restated in this particular setting using $S$ and $G$. Let $S_G$ be a webbed $\Cus$-semigroup. We have $\nu_+(S_G)\simeq S$ and hence, $S_G$ has maximal element if and only if $S$ has a largest element $\infty_S$. In that case, we have that  $\nu_{\max}(S_G)\simeq G(\infty_s)$. Moreover, any element in $\Lat(S_G)$ is of the form $I_G$ where $I\in\Lat(S)$ and $G$ is the restriction of $G$ to $I$. Let $\alpha_\eta$ be a webbed $\Cus$-morphism and let $I,J$ be singly-generated $\Cus$-ideals of $S,T$ respectively, such that $J\supseteq \alpha(I)$. Then $(\alpha_{\mid IJ})_{max}:=H(\alpha(e_{I_{max}})\leq e_{J_{max}})\circ \eta_{e_{I_{max}}}$. 

A \emph{concrete $\Cus$-semigroup} is a $\Cus$-semigroup arising from a $\CatCa$-algebra via some functor $F:\CatCa\longrightarrow \Cus$. Similarly a \emph{concrete $\Cus$-morphism} is a $\Cus$-morphism arising from a $^*$-homomorphism via $F$. In the sequel, we provide a generic method to produce invariants for $\CatCa$-algebras via the webbing functor. Consequently, concrete $\Cus$-semigroups we consider are a fortiori webbed $\Cus$-semigroups and we shall transpose the ideal-quotient theory in greater details for these concrete $\Cus$-semigroups in the upcoming section.
\end{prg}

\section{Applications and Outlook}
This section is aiming to generically create invariants for $\CatCa$-algebras using the Cuntz semigroup together with other existing invariants such as (algebraic or total) $\K$-Theory via the webbing functor. We establish properties such as exactness or continuity of these invariants and using the ideal and quotients theory for $\Cus$-semigroups, we are able to unravel information they contain about a given $\CatCa$-algebra. In a second stage, we show that many of existing invariants for $\CatCa$-algebras can be rewritten as a well-suited $\Cu_\K$-type generic construction. As an application, we are constructing an invariant that incorporates the information of the Cuntz semigroup together with the information of the Hausdorffized algebraic $\K_1$-group and we expose the benefits that could follow for classification of $\CatCa$-algebras. For that matter, we will be led to consider distance between concrete and webbed $\Cus$-semigroups. 

\subsection{The \texorpdfstring{$\Cu_\K$}{CuK}-type invariants}\label{sec:A}  In this first part, we provide a generic method to build invariants for $\CatCa$-algebras. Other types generic constructions could be considered in the future. (E.g. using the double dual of the Cuntz semigroup as a \textquoteleft base space\textquoteright\ instead of the Cuntz semigroup itself, or use another category than $\AbGp$ for the \textquoteleft fibers\textquoteright. See comments in the next subsection.) 
The \emph{$\Cu_\K$-type construction} we expose here roughly consists of mixing the functor $\Cu:\CatCa\longrightarrow \Cu$ together with a continuous functor $\K:\CatCa\longrightarrow \AbGp$ via the webbing functor. Recall that for any $\CatCa$-algebra $A$ the lattice of ideals of $A$ is in bijection with the lattice of ideals of $\Cu(A)$. For notational purposes, we are indistinguishably using $I_x$ where $x\in \Cu(A)$ to also refer to an element of $\Lat(A)$, when we should invoke $I_{a_x}$ where $a_x\in (A\otimes\mathcal{K})_+$ such that $[a_x]=x$.

Let $\K:\CatCa\longrightarrow \AbGp$ be a functor that preserves inductive limits and such that $\K(0)=\{e\}$. For any $\CatCa$-algebra $A$, we construct a functor $\K_A$ defined by 
\[
\begin{array}{ll}
	\K_A: \Cu(A)\longrightarrow \AbGp \\
	\hspace{1,6cm} x\longmapsto \K(I_x)\\
	\hspace{0,95cm}x\leq y\longmapsto \K(I_x\overset{\subseteq}\longrightarrow I_y)
\end{array}
\] 
Thus we can consider the $\Cus$-semigroup $\Cu_\K(A):=\mathfrak{w}(\Cu(A),\K_A)$.  

Let $\phi:A\longrightarrow B$ be a $^*$-homomorphism. We construct a natural transformation $\K_\phi:\K_A\Rightarrow \K_B\circ\Cu(\phi)$ defined by 
\[
\begin{array}{ll}
	(\K_{\phi})_x:\K_A(x)\longrightarrow \K_B(\Cu(\phi)(x)) \\
	\hspace{1,7cm}[u]\longmapsto [\phi(u)]
\end{array}
\]
In other words, $(\K_{\phi})_x:=\K(\phi_{\mid}: I_x\longrightarrow I_{\Cu(\phi)(x)})$ where $\phi_{\mid}$ is obtained by restriction of $\phi$. It follows that $\K_\phi$ is a well-defined transformation and we can consider the $\Cus$-morphism $\Cu_\K(\phi):=\mathfrak{w}(\Cu(\phi),\K_\phi)$. 
We hence define a functor $\Cu_\K$ as follows.
\[
\begin{array}{ll}
	\Cu_\K: \CatCa\longrightarrow \Cus \\
	\hspace{1,15cm} A\longmapsto \Cu_\K(A)\\
	\hspace{1,2cm}\phi \longmapsto \Cu_\K(\phi)
\end{array}
\] 

In the sequel, we study functorial properties of any $\Cu_\K$-type invariant (e.g. continuity, exactness) and its ideal theory (e.g. commutative diagram with exact rows compatible with concrete $\Cus$-morphisms). From now on, we automatically assume that any functor $\K:\CatCa\longrightarrow \AbGp$ satisfies $\K(0)=\{e\}$. 
A first handy feature is to explicitly rewrite concrete $\Cus$-semigroups and concrete $\Cus$-morphisms arising from a $\Cu_\K$-type construction in terms of the original functors $\Cu$ and $\K$.

Consider a $\Cu_\K$-type construction and let $A$ be a $\CatCa$-algebra. Then
\[
\Cu_\K(A)=\bigsqcup\limits_{I\in \Lat_f(A)} \Cu_f(I)\times \K(I)
\]
where $\Cu_f(I):=\{x\in \Cu(A) \mid I_x=I\}$. 

Let $B$ be another $\CatCa$-algebra and let $\phi:A\longrightarrow B$ be a $^*$-homomorphism. Then
\[
\K_\phi=\{\K(I\longrightarrow I_\phi)\}_{I\in \Lat_f(A)}
\]
where $I_\phi:=\min\{J\in \Lat(B)\mid J\supseteq \phi(I)\}$ which is an element of $\Lat_f(B)$. 

Therefore we compute 
\[
\Cu_\K(\phi)(x,g)=(\Cu(\phi)(x),\K(I_x\longrightarrow I_{\Cu(\phi)(x)})(g)) \text{ for any } (x,g)\in \Cu_\K(A).
\]
\begin{prop}
Let $A$ be a $\CatCa$-algebra and consider a $\Cu_\K$-type invariant. Then 

(i) $\Cu_\K(A)$ satisfies axiom (05). 

(ii) $\Cu_\K(A)$ satisfies axiom (PWC) whenever $\Cu(A)$ has weak cancellation.

(iii) $\Cu_\K(A)$ is almost divisible whenever $\Cu(A)$ is almost divisible and $\K_A$ is stable.
\end{prop}
\begin{proof}
This directly follows from \autoref{prop:webprop}.
\end{proof}

Another significant benefit of the work done in the previous parts is that we are able to show that any $\Cu_\K$-type invariant is continuous given that functor $\K:\CatCa\longrightarrow \AbGp$ is continuous.

\begin{thm}
Let $\K:\CatCa\longrightarrow \AbGp$ be a continuous functor with respect to inductive limits and consider $\Cu_\K:\CatCa\longrightarrow \Cus$ as constructed above. 

Then $\Cu_\K$ is a continuous functor with respect to inductive limits. 
\end{thm}

\begin{proof}
Let $(A_i,\phi_{ij})_{i\in I}$ be an inductive system in the category $\CatCa$ with limit $(A,\phi_{i\infty})_i$. We know that the webbing functor $\mathfrak{w}:\CuG\longrightarrow \Cus$ is continuous and hence, it suffices to show that the Cuntz group systems map $(\beta,\mu)$ (obtained by the universal property of the following induced system in $\CuG$) is a $\CuG$-isomorphism.
\[
\xymatrix{
(\Cu(A_i),\K_{A_i})\ar[dd]_{(\Cu(\phi_{ij}),\K_{\phi_{ij}})}\ar[rd]^{(\alpha_{i\infty},\nu_{i\infty})}\ar@/^2pc/[rrd]^{(\Cu(\phi_{i\infty}),\K_{\phi_{i\infty}})} &  &\\
& (S_\infty,\K_\infty)\ar@{.>}[r]_{\hspace{-0,4cm}\exists!\, (\beta,\mu)}&  (\Cu(A),\K_A)\\
(\Cu(A_j),\K_{A_j})\ar[ru]_{(\alpha_{j\infty},\nu_{j\infty})}\ar@/_2pc/[rru]_{(\Cu(\phi_{j\infty}),\K_{\phi_{j\infty}})}  & &
} 
\]
From the explicit construction of \autoref{thm:limG}, we observe that $(S_\infty,\alpha_{i\infty})$ is the $\Cu$-limit of the system $(\Cu(A_i),\Cu(\phi_{ij}))_i$. Therefore $\beta: S_\infty\overset{\id}\simeq \Cu(A)$ is the canonical isomorphism. We next have to show that the natural transformation $\mu:\K_\infty\Rightarrow \K_A\circ \beta$ is in fact a natural isomorphism. Let $x\in \Cu(A_i)$ and write $x_\infty:=\Cu(\phi_{i\infty})(x)$. We have that $\beta(x_\infty)=x_\infty$ and the following diagram commutes.
\[
\xymatrix{
\K_{A_i}(x)\ar[rr]^{(\nu_{i\infty})_{x_i}}\ar@/^{-2pc}/[rrrr]_{(\K_{\phi_{i\infty}})_{x_i}}&& \K_\infty(x_\infty)\ar@{-->}[rr]^{\mu_{x_\infty}} && \K_A(x_\infty)
} 
\]
Rewriting everything in simpler terms, using the definition of $\K_\phi$ and the explicit construction in \autoref{thm:limG}, it all comes down to 
\[
\left\{
\begin{array}{ll}
\K_\infty(x_\infty):=\lim\limits_{\longrightarrow}( \K(I_x),\K(\phi_{ij})  )\\
\K_A(x_\infty):=\K(\lim\limits_{\longrightarrow}( I_x,\phi_{ij}  ))
\end{array}
\right.
\]
where $\phi_{ij}$ denotes the restriction map to the ideal $I_x$ of $A_i$. We conclude using the continuity of the functor $\K:\CatCa\longrightarrow \AbGp$.
\end{proof}

We end up this part by highlighting the ideal-quotient theory of $\Cu_\K$-type invariants. As stated in the last observation of the previous section, the ideal-quotient theory of $\Cu_\K(A)$ will be closely related to the one of $\Cu(A)$ and to $\K(A)$. We start by studying the lattice of ideals, their positive cone and when they exist, their maximal elements. The following results are directly deduced from \autoref{prop:idealppty}, \cite[Theorem 3.21]{C21b} and \autoref{thm:diagrammorph}.

\begin{thm}
Let $A$ be a $\CatCa$-algebra and consider a $\Cu_\K$-type construction. Then there exist complete lattice isomorphisms given by 
\[
\begin{array}{ll}
\Lat(A)\simeq \Lat(\Cu(A))\simeq \Lat(\Cu_\K(A))\\
\hspace{0,8cm}I\longmapsto \hspace{0,3cm} \Cu(I) \hspace{0,3cm}\longmapsto \Cu_\K(I)
\end{array}
\]
Moreover, these bijective correspondences match the (sub)sets of singly generated ideals of all three lattices, i.e. $\Lat_f(A)\simeq \Lat_f(\Cu(A))\simeq \Lat_f(\Cu_\K(A))$.
\end{thm}

\begin{cor}
Consider a $\Cu_\K$-type construction. Let $A$ be a $\CatCa$-algebra and let $I\in \Lat(A)$. The following are equivalent

(i) $I\in \Lat_f(A)$.

(ii) $\Cu(I)$ has a largest element $\infty_I$.

(iii) $\Cu_\K(I)$ has maximal elements.\\
In that case, we compute that $\nu_{\max}(\Cu_\K(I))\simeq \K(I)$.
\end{cor}

\begin{thm}
Consider a $\Cu_\K$-type construction. Let $A,B$ be $\sigma$-unital $\CatCa$-algebras and let $\phi:A\longrightarrow B$ be a $^*$-homomorphism. Let $I\in\Lat_f(A)$.

 Then the following diagram is commutative with exact rows in the category $\Cus$.\\
\[
\xymatrix@!R@!C@R=20pt@C=0pt{
&0\ar[rr] &&  \Cu(I_\phi)\ar[rr]\ar@{.>}[dd] &&  \Cu_\K(I_\phi)\ar[rr]\ar@{.>}[dd] &&  \K(I_\phi)\ar@{.>}@/_{-1.2pc}/[ll]\ar[rr]\ar@{.>}[dd] && 0 \\
0\ar[rr] && \Cu(I)\ar[rr]\ar[dd]\ar[ur]_{\hspace{0,1cm}\Cu(\phi_\mid)} && \Cu_\K(I)\ar[rr]\ar[dd]\ar[ur]_{\Cu_\K(\phi_\mid)} && \K(I)\ar@/_{-1.2pc}/[ll]\ar[rr]\ar[dd]\ar[ur]_{\K(\phi_\mid)} && 0 \\
&0\ar@{.>}[rr] && \Cu(B)\ar@{.>}[rr] && \Cu_\K(B)\ar@{.>}[rr] && \K(B)\ar@{.>}@/_{-1.2pc}/[ll]\ar@{.>}[rr] && 0\\
0\ar[rr]{} && \Cu(A)\ar[rr]\ar@{.>}[ur]_{\Cu(\phi)} && \Cu_\K(A)\ar[rr]\ar@{.>}[ur]_{\Cu_\K(\phi)} && \K(A) \ar@/_{-1.2pc}/[ll]\ar[rr]\ar@{.>}[ur]_{\K(\phi)} && 0
}
\vspace{1cm}\] 
where $\phi_\mid:I\longrightarrow I_\phi$ is the restriction of $\phi$, the horizontal arrows are the respective maps $i,j,q$ as in \autoref{thm:splitmax} and the vertical arrows are $\Cu(I\overset{\subseteq}\longrightarrow A), \Cu_\K(I\overset{\subseteq}\longrightarrow A), G(I\overset{\subseteq}\longrightarrow A)$ and similarly with $I_\phi\overset{\subseteq}\longrightarrow B$.
\end{thm}

\begin{rmk}
(i) Observe that $\K(\phi)=\K(A_\phi \overset{\subseteq}\longrightarrow B)\circ\K(A\longrightarrow A_\phi)$. 

(ii) We have assumed that the $\CatCa$-algebras $A,B$ are $\sigma$-unital to ensure that $A\in\Lat_f(A)$ and $B\in\Lat_f(B)$. For the general case an analogous result applies replacing $A,B$ by singly-generated ideals $I',J$ that respectively contain $I,I_\phi$ and such that $J\supseteq I_{\phi}$. 
\end{rmk}

We next show that under reasonable assumptions on the original functor $\K:\CatCa\longrightarrow \AbGp$ then $\Cu_\K$-type invariants are exact. To this end, we recall the notion of half-exactness for functors and define a surjectivity property.

\begin{dfn}
Let $\K:\mathcal{C}\longrightarrow \mathcal{D}$ be a functor between categories with exact sequences. 

(i) We say that $\K$ is \emph{half-exact} (respectively \emph{exact}) if for any short-exact sequence $0\longrightarrow I \longrightarrow A \longrightarrow Q \longrightarrow 0$ we have that $\K(I)\longrightarrow \K(A)\longrightarrow \K(Q)$ is exact (respectively $0\longrightarrow \K(I)\longrightarrow \K(A)\longrightarrow \K(Q)\longrightarrow 0$ is exact.)

(ii) We say that $\K$ has the \emph{surjectivity property} if for any exact sequence $A \longrightarrow Q \longrightarrow 0$ we have that $ \K(A)\longrightarrow \K(Q)\longrightarrow 0$ is exact.
\end{dfn}

\begin{thm}
Let $\K:\CatCa\longrightarrow \AbGp$ be an half-exact functor and consider $\Cu_\K:\CatCa\longrightarrow \Cus$ as constructed above. Let $A$ be a $\CatCa$-algebra and let $I\in\Lat(A)$. 

(i) There exists an order-embedding $\Cus$-morphism $\alpha:\Cu_\K(A)/\Cu_\K(I)\lhook\joinrel\longrightarrow  \Cu_\K(A/I)$. 

(ii) If moreover $\K$ has the surjectivity property, then $\alpha$ is a $\Cus$-isomorphism. In other words \[\Cu_\K(A)/\Cu_\K(I)\simeq  \Cu_\K(A/I).\]
\end{thm}

\begin{proof}
The proof is almost identical as in \cite[Theorem 4.5]{C21b} but we give the main lines for the sake of completeness. Let $0\longrightarrow I \overset{i}\longrightarrow A \overset{\pi}\longrightarrow A/I \longrightarrow 0$ be an ideal exact sequence. We first show that $\Cu_\K(\pi)(x,k)\leq \Cu_\K(\pi)(y,l)$ if and only if there exists $(z,m)\in \Cu_\K(I)$ such that $(x,k)\leq (y,l)+(z,m)$. The converse implication is almost immediate from that fact that $\Cu_\K(\pi)(\Cu_\K(I))=\{0\}$. Now let $(x,k),(y,l)\in \Cu_\K(A)$ be such that $\Cu_\K(\pi)(x,k)\leq \Cu_\K(\pi)(y,l)$. Arguing similarly as in \cite{C21b}, we deduce that there exists $z\in \Cu(A)$ such that $x\leq y+z$ and we obtain the following commutative diagram
\[
\xymatrix{
&&\K(x)\ar[rr]^{(\K_\pi)_x}\ar[d]_{\K_A(\leq)} &&\K(\overline{x})\ar[d]^{\K_A(\leq)}   \\
\K(z)\ar[rr]_{\K_A(\leq)}&&\K(y+z)\ar[rr]_{(\K_\pi)_{(y+z)}} && \K(\overline{y})
}
\]
Moreover, by half-exactness of $\K$, the bottom-row is exact and hence,  $\ker((\K_\pi)_{(y+z)})\simeq \K(z)$. We also know by hypothesis that $[\K_A(x\leq y+z)(k)]=[\K_A(y\leq y+z)(l)]$ in $\K(\overline{y})$. We deduce that there exists $m\in \K(z)$ such that $\K_A(x\leq y+z)(k)=\K_A(y\leq y+z)(l)+ \K(z\leq y+z)(m)$ from which the equivalence is established. We can now use \autoref{prop:idealquotient} (iii) to conclude that the map $\Cu_\K(\pi)$ factorizes through $\Cu_\K(I)$ to give rise to an order-embedding $\alpha:\Cu_\K(A)/\Cu_\K(I)\lhook\joinrel\longrightarrow  \Cu_\K(A/I)$. Finally, under the assumption of the surjectivity property, one can argue similarly as in \cite{C21b} to get that $\Cu_\K(\pi)$ is surjective. Therefore, by \autoref{prop:idealquotient} (iii) so is $\alpha$.
\end{proof}

\begin{cor}
Let $\K:\CatCa\longrightarrow \AbGp$ be an half-exact functor with the surjectivity property. Then the functor $\Cu_\K:\CatCa\longrightarrow \Cus$ is exact. 

In particular, $0\longrightarrow \Cu_\K(I) \longrightarrow \Cu_\K(A) \longrightarrow \Cu_\K(A/I) \longrightarrow 0$ is a short-exact sequence in $\Cus$ for any $\CatCa$-algebra $A$ and any ideal $I\in \Lat(A)$.
\end{cor}

\begin{proof}
We know that any short exact sequence in $\CatCa$ is isomorphic to an ideal-quotient short exact sequence. Now combine \autoref{prop:idealquotient} (ii) with the above theorem to get the result.
\end{proof}

\begin{qst}
It has been shown in \cite[4.4]{C21a} that $\Cu_\K$-type invariants does not pass pullbacks in general. Can we find a condition on $\K$ under which the functor $\Cu_\K$ passes pullbacks?
\end{qst}

\subsection{Invariants for \texorpdfstring{$\CatCa$}{C*}-algebras - Hausdorffized unitary Cuntz semigroup}
In this part, we use the generic $\Cu_\K$-type construction to give an insight on invariants for $\CatCa$-algebras. We also suggest other generic constructions that could be considered for future lines of work.

\vspace{0,2cm}$\hspace{-0,34cm}\bullet\,\,\textbf{The unitary Cuntz semigroup and its total version}$. 
The present author has incorporated the information contained in the $\K_1$-group into the original Cuntz semigroup which has led to the \emph{unitary Cuntz semigroup} $\Cu_1$. See \cite{C21a}. Later on, a similar approach has been used in \cite{AL23} aiming to incorporate the total $\K$-theory $\underline{\K}$ which has led to the \emph{total Cuntz semigroup} $\underline{\Cu}$. Both constructions are very alike and hence share many categorical properties. In fact, many of the proofs done in \cite{AL23} are directly drawn from the ones in \cite{C21a}, even though everything had to be re-proven with care. Using our methods, we obtain all the categorical results contained in both papers by simply (re)defining these invariants as $\Cu_\K$-type constructions. More particularly, we define
\[
\Cu_{\K_1}:\CatCa\longrightarrow \Cus \hspace{1cm}\text{ and }\hspace{1cm} \Cu_{\underline{\K}}:\CatCa\longrightarrow \Cus
\]
Observe that both of the constructions are now well-defined outside the stable rank one setting, which is a significant benefit. Moreover, we get a neat codomain category, continuity, exactness (only in the stable rank on setting), ideal-quotient theory and as we will see in a moment, distances between morphisms. 

\vspace{0,2cm}$\hspace{-0,34cm}\bullet\,\,\textbf{The $\underline{\KT}_u$ invariant and the Hausdorffized unitary Cuntz semigroup}$. We next tackle the most notorious invariant for $\CatCa$-algebras, which is commonly known as the \emph{Elliott invariant}. Over the decades and the work of many hands it has proven that the original Elliott invariant $\Ell$ is a complete invariant for the class of simple unital $\mathcal{Z}$-stable nuclear separable $\CatCa$-algebras classifying the UCT assumption. Furthermore, in the aim of classifying non-simple $\CatCa$-algebras and $^*$-homomorphisms, the original Elliott invariant has been refined several times and we denote its most sophisticated form by $\underline{\KT}_u$. (See \cite{W23}.)
The invariant $\underline{\KT}_u$ is built out of the total $\K$-Theory together with traces and the Hausdorffized algebraic $\K_1$-group and compatibility conditions. These compatibility conditions and the inherently distinct nature of the objects shaping $\underline{\KT}_u$ makes it hard for a clean categorical context. On the other hand, $\underline{\KT}_u$ has also been successfully used to classify non-simple $\CatCa$-algebras. (See e.g. \cite{GJL20}.) However, whenever dealing with non-simple $\CatCa$-algebras, one has to take into account $\underline{\KT}_u$ of all the ideals (which can be plenty) together with many more compatibility conditions. Again, this is far from being convenient from a categorical perspective. 
We now attempt to offer a well-defined and well-structured invariant in the category $\Cus$ that contains the same information as $\underline{\KT}_u$ (leaving the total $\K$-Theory aside for the moment) and such that all \textquoteleft compatibility conditions\textquoteright\ are inherently satisfied. More specifically, we incorporate the Hausdorffized algebraic $\K_1$-group into the Cuntz semigroup via a $\Cu_\K$-type construction.

\begin{prg}\textbf{The Hausdorffized unitary Cuntz semigroup}.
Let $A$ be a $\CatCa$-algebra. Recall that the \emph{Hausdorffized algebraic $\K_1$-group} $\overline{\K}^{\alg}_1:\CatCa\longrightarrow \AbGp$ is a well-defined continuous functor that sends $A\longmapsto U^\infty(A)/\overline{{DU}^\infty(A)}$. (We take the unitization of $A$ if needed.) We refer the reader to \cite{T95} and \cite{NT96} for a systematic study of this $\K$-theoretical invariant.

The \emph{Hausdorffized unitary Cuntz semigroup} is defined as \[\Cu_{\overline{\K}^{\alg}_1}:\CatCa\longrightarrow\Cus \] In particular, the Hausdorffized unitary Cuntz semigroup of $A$ can be computed as follows
\[
\Cu_{\overline{\K}^{\alg}_1}(A)=\bigsqcup\limits_{I\in \Lat_f(A)} \Cu_f(I)\times \overline{\K}^{\alg}_1(I)
\] 

From the results obtained in the previous part, we know that $\Cu_{\overline{\K}^{\alg}_1}: \CatCa\longrightarrow \Cus$ is a well-defined continuous functor.  Many more properties can be immediately deduced (e.g. ideal-quotient theory) and they will be explicitly exposed in a future work. We point out that at the moment of writing, it is not clear whether $\overline{\K}^{\alg}_1$ is half-exact or not, nor if it has the surjectivity property upon restriction some $\CatCa$-algebras. (In general, this shall not hold since unitary elements in quotients need not lift.) This shall be explored in a  future investigation where we plan to take on the conjecture stated in \cite{C23b} which stipulates that the Hausdorffized unitary Cuntz semigroup classifies unitary elements of unital $\AH_1$-algebras and more generally of unital $\ASH_1$-algebras with torsion-free $\K_1$-group. We highlight that this conjecture has been based on the fact this invariant is inherently containing the information and compatibility conditions of $\underline{\KT}_u$ (a variation of it that does not contain the total $\K$-Theory to be precise) and the fact that $\overline{\K}^{\alg}_1$ has been used in the classification of simple $\A\!\T$-algebras done in \cite{NT96}. Additionally, the de la Harpe-Skandalis determinant (closely related to $\overline{\K}^{\alg}_1$ via the Nielsen-Thomsen sequence) seems to be indispensable to classify unitary elements outside the $\AF$ setting. See \cite[4A-5B]{C23b}. 
\end{prg}

$\hspace{-0,34cm}\bullet\,\,\textbf{Other types of construction}$. 
The $\Cu_\K$-type generic method we have provided above only relies on the webbing functor that webs a \textquoteleft base space\textquoteright, i.e. a $\Cu$-semigroup $S$, with \textquoteleft fibers\textquoteright, i.e. a functor $G\in\Fun(S,\AbGp)$, to produce a $\Cus$-semigroup. Consequently, more generic methods could arise from this categorical tool. Let us give a couple of open lines of investigation.

\vspace{0,2cm}\textbf{[An invariant for Graph $\CatCa$-algebras in the category $\Cus$]} Graph $\CatCa$-algebras have been studied and classified for decades now. Most of the classification results obtained are using (variation of) filtered $\K$-Theory which roughly consists of the information contained in 6-terms $\K$-theoretical exact sequences induced by all the ideal-quotient exact sequences of the $\CatCa$-algebras. See e.g. \cite{R97} and \cite{ERRS21}, \cite{ERRS22} for state of the art results on the matter. Nevertheless, the filtered $\K$-Theory also lacks of categorical consistency, in the sense that the objects are defined to be multiple hexagonal cyclic exact sequences (in abelian groups) together with compatibility conditions. To remedy this issue, one could think of introducing the category $\varhexagon$ of hexagonal cyclic exact sequences of abelian groups and define similarly Cuntz hexagonal systems to be pairs $(S,G)$ where $S\in \Cu$ and $G\in \Fun(S,\varhexagon)$. From that point on, we would be able to construct a webbing functor $\mathfrak{w}:\Cu_{*\varhexagon}\longrightarrow \Cus$ which transforms a pair $(S,G)$ into a semigroup $S_G:=\{(s,\mathfrak{g})\mid s\in S,\mathfrak{g}\in G_1\times\dots\times G_6\}$ where $G_1,\dots, G_6$ are the groups of the hexagonal cyclic sequence $G(s)$. The sum and order would be defined similarly as done previously and since the limits in the category $\varhexagon$ are algebraic, all of the above theory would work similarly. In particular, it shall be true that the webbing functor is continuous and hence, that the functor $\Cu_{\varhexagon}:=\mathfrak{w}(\Cu(A),{\K_{\varhexagon}}_A)$, where ${\K_{\varhexagon}}_A:\Cu(A)\longrightarrow \varhexagon$ sends $x\longmapsto\K_{\varhexagon}(I_x)$, shall be well-defined and continuous from $\CatCa$ to $\Cus$.

\begin{qst}
Is a \textquoteleft filtered Cuntz semigroup\textquoteright\ $\Cu_{\varhexagon}$ able to extend some classification result obtained for graph $\CatCa$-algebras?
\end{qst}

\textbf{[A Nielsen-Thomsen sequence in the category $\Cus$]} It has been noticed in \cite{ERS09}  that the double dual $\L(F(S))$ of a $\Cu$-semigroup $S$, which roughly consists of a specific subset of lower-semicontinuous functions on the cone of functionals $F(S)$, is itself a $\Cu$-semigroup. In fact, it has been proven in \cite{R13} that $\L(F(S))\simeq S\otimes[0,\infty]$. Furthermore, the double dual of $\Cu(A)$ for a given $\CatCa$-algebra $A$ is closely related to the affine functions on tracial simplex $T(A)$ of $A$. Finally, the Nielsen-Thomsen sequence considered e.g. in \cite{NT96} is linking the latter space with the Hausdorffized algebraic $\K_1$-group and the usual $\K_1$-group of $A$. An idea to pursue would consist in replacing the Cuntz semigroup as base space by its double dual. Nevertheless, at the moment of writing, it is not clear how one could mimick $\Cu_\K$-type constructions with the base space being $\L(F(\Cu(A)))$. We pose the following question, related to \cite[Question 16.8]{GP23}.
\begin{qst}
Is there a natural way to define a functor $\LF\!\K_A: \L(F(\Cu(A)))\longrightarrow \AbGp$ from an existing invariant $\K:\CatCa\longrightarrow \AbGp$ as done for $\Cu_\K$-type constructions? If applicable, could we obtain an analogous Nielsen-Thomsen sequence in the category $\Cus$ via the split-exact sequence of \autoref{thm:splitmax} applied to the $\Cus$-semigroup $\mathfrak{w}(\L(F(\Cu(A))),\LF\!\K_A)$?
\end{qst}

\subsection{Distance between \texorpdfstring{$\Cus$}{Cu*}-morphisms}
We aim to define a semimetric on the set of $\Cus$-morphisms between webbed $\Cus$-semigroups and a fortiori between concrete $\Cus$-semigroups. Such a semimetric can be defined whenever the base spaces $S,T$ of the given webbed $\Cus$-semigroups $S_G,T_H$ admit a semimetric on $\Hom_{\Cu}(S,T)$. We first start by introducing an abstract notion comparison of $\Cus$-morphisms modeled after the one defined for $\Cu$-semigroups. 
\begin{dfn}
Let $S,T$ be $\Cus$-semigroups and let $\Lambda\subseteq S$. Consider maps $\alpha,\beta:\Lambda\longrightarrow T$ preserving the order and the compact-containment relation.

(i) We say that $\alpha$ and $\beta$ \emph{compare on $\Lambda$} and we write $\alpha\underset{\Lambda}{\simeq}\beta $ if, for any $g,h\in\Lambda$ such that $g\ll h$ (in $S$), we have that $\alpha(g)\leq \beta(h)$ and $\beta(g)\leq \alpha(h)$.

(ii) We say that $\alpha$ and $\beta$ \emph{strictly compare on $\Lambda$} and we write $\alpha\underset{\Lambda}{\approx}\beta $ if, for any $g,h\in\Lambda$ such that $g\ll h$ (in $S$), we have that $\alpha(g)\ll \beta(h)$ and $\beta(g)\ll \alpha(h)$.
\end{dfn}

We now focus on the specific case where the domain webbed $\Cu$-semigroup $S_G$ is such that $S\simeq \Lsc(\T,\overline{\N})$. We operate analogously as the work done for $\Cu$-semigroups and we refer the reader to \cite[5.1]{C22}. (Or \cite[2.1]{C23b} for a summary.) More specifically, for any $n\in\N$ we consider an equidistant partition $(x_k)_0^{2^n}$ of $\T$ of size $\frac{1}{2^n}$ with $x_0=x_{2^n}$. For any $k\in \{1,\dots,2^n\}$, we define the open interval $U_k:=]x_{k-1};x_k[$ of $\T$. It is immediate to see that $\{\overline{U_k}\}_1^{2^n}$ is a (finite) closed cover of $\T$. Now consider
\[
\Lambda^*_n:=\{(s,g)\mid s\in \Lambda_n \text{ and } g\in G(s)\}.
\]
where $\Lambda_n:=\{f\in\Lsc(\T,\{0,1\})\mid f_{|U_k} \text{ is constant for any }k\in \{1,\dots, 2^n\}\}$. 
\begin{dfn}
Let $S_H,T_H$ be webbed $\Cu$-semigroups with $S\simeq \Lsc(\T,\overline{\N})$. Let $\alpha,\beta:S_G\longrightarrow T_H$ be $\Cus$-morphisms. We define 
\[ 
dd_{\Cus}(\alpha,\beta):= \inf\limits_{n\in\N} \{ \frac{1}{2^n} \mid \alpha\underset{\Lambda^*_n}{\simeq}\beta \}.
\]  
If the infimum does not exist, we set the value to $\infty$.
\end{dfn}

This is a well-defined semimetric that we refer to as the \emph{discrete $\Cus$-semimetric}. 
We remark that one can also build a $\Cus$-metric in the following way.
\[d_{\Cus}(\alpha,\beta):=\inf \{r>0\mid \forall (1_U,g)\ll(1_{U_r},h) \text{ then } \alpha(1_U,g),\beta(1_U,g)\leq \alpha(1_{U_r},h),\beta(1_{U_r},h)\} \]
where $U$ is any open set of $\T$ and $U_r:=\underset{x\in U}{\cup}B_r(x)$. 

This is a well-defined metric on $\Hom_{\Cus}(S_H,T_H)$ which is equivalent to the discrete $\Cus$-semimetric. More specifically, we have $dd_{\Cus}\leq d_{\Cus}\leq 2dd_{\Cus}$ and hence, $dd_{\Cus}$ satisfies a weak form of triangular inequality, sometimes referred to as the 2-relaxed triangular inequality.

Furthermore, let us point out that the $\Cus$-semimetric construction can be replicate whenever the base space $S$ of the domain $\Cus$-semigroup $S_G$ has a uniform basis. E.g. whenever $S\simeq \Lsc(X,\overline{\N})$ where $X$ is any compact one-dimensional $\CW$-complex. We have decided not to pursue further in this direction yet, as we are driven by the classification of unitary elements at the moment of writing.  

Lastly, we point out that in the particular case where the $\Cus$-morphisms are in fact webbed $\Cus$-morphisms, we are able to give a better interpretation of the $\Cus$-semimetric in terms of the base spaces and fibers involved.

\begin{prop}
Let $S_G,T_H$ be webbed $\Cu$-semigroups with $S\simeq \Lsc(\T,\overline{\N})$. Consider two  webbed $\Cus$-morphisms $\alpha_\eta,\beta_\nu:S_G\longrightarrow T_H$. Then $dd_{\Cus}(\alpha_\eta,\beta_\nu)\leq \frac{1}{2^n}$ if and only if 

(i) $dd_{\Cu}(\alpha,\beta)\leq \frac{1}{2^n}$

(ii) For any $s,t\in \Lambda_n$ such that $s\ll t$ (in $S$), the diagram $D(s,t)$ is commutative
\[
\xymatrix@!R@!C@R=30pt@C=30pt{
H\circ \beta(s)\ar[d]_{H(\leq)} & G(s)\ar[l]_{\nu_{\beta(s)}}\ar[r]^{\eta_{\alpha(s)}}\ar[d]_{G(\leq)} & H\circ\alpha(s)\ar[d]^{H(\leq)}\\
H\circ\alpha(t) & G(t)\ar[l]^{\eta_{\alpha(t)}}\ar[r]_{\nu_{\beta(t)}} & H\circ \beta(t)\\
}
\]
\end{prop}

We end this manuscript with some hints about future endeavors that could be pursued around the $\Cus$-semimetric.

\vspace{0,2cm}\textbf{[Approximate Intertwinings in the category $\Cus$]} In \cite{C23}, the author has exhibited two non-simple $\CatCa$-algebras that agree at level of $\Cu$ and $\K_1$ and yet are distinguished by their unitary Cuntz semigroup. The approximate intertwining theorem for the category $\Cu$ obtained in \cite[Theorem 3.16]{C22} turned out to be very handy when came to proving that the $\CatCa$-algebras constructed had isomorphic Cuntz semigroups. Developing an analogous approximate intertwining theorem in the category $\Cus$ would be a key tool to construct in a similar spirit two $\CatCa$-algebras that agree at level of their unitary Cuntz semigroup $\Cu_{\K_1}$ and yet are distinguished by their Hausdorffized unitary Cuntz semigroup $\Cu_{\overline{\K}^{\alg}_1}$.

\vspace{0,2cm}\textbf{[Classification of unitary elements some of $\ASH_1$-algebras]} As stated in the conjecture at the beginning of this paper and in \cite{C23b}, the Hausdorffized unitary Cuntz semigroup is likely to become a powerful invariant in the classification of unitary elements of a large class of stable rank one $\CatCa$-algebras. For instance, we should be able to distinguish the unitary elements exposed in \cite[4A - 5B]{C23b} by means of $\Cu_{\overline{\K}^{\alg}_1}$ and the $\Cus$-semimetric could appear handy in this task. We highlight that the $\Cus$-semimetric has a broader and more crucial role to play in the conjecture, in a similar fashion as the $\Cu$-semimetric came into play in the classification of unitary elements of $\AF$-algebra by means of the Cuntz semigroup. (See  \cite[Theorem 3.6]{C23b}.)

\end{document}